\documentclass[doublecolumn]{IEEEtran}
\usepackage{amssymb}
\usepackage{graphicx}
\usepackage{amsmath}
\usepackage{amsfonts}

\graphicspath{{"C:/Users/mnagh_000/Google Drive/papers/Journals/2018/decentralized/"},{"D:/Personal Google Drive/papers/Journals/2018/decentralized/"}}
\newtheorem{theorem}{Theorem}

\newtheorem{algorithm}{Algorithm}

\newtheorem{corollary}{Corollary}

\newtheorem{definition}{Definition}
\newtheorem{example}{Example}

\newtheorem{lemma}{Lemma}

\newtheorem{proposition}{Proposition}
\newtheorem{remark}{Remark}

\newtheorem{assumption}{Assumption}
\newenvironment{proof}[1][Proof]{\noindent\textbf{#1.} }{\ \rule{0.5em}{0.5em}}
\begin{document}

\title{{A Youla Operator State-Space Framework for Stably Realizable
Distributed Control}}
\author{Mohammad Naghnaeian, Petros G. Voulgaris, and Nicola Elia \thanks{%
M. Naghnaeian is with the Mechanical Engineering Department, Clemson
University, Clemson, SC, USA \texttt{\small mnaghna@clemson.edu}} \thanks{%
P. G. Voulgaris is with the Aerospace Engineering Department and the
Coordinated Science Laboratory, University of Illinois, Urbana, IL, USA 
\texttt{\small voulgari@illinois.edu}, and Aerospace Engineering Department,
Khalifa University, Abu Dhabi, UAE \texttt{\small %
petros.voulgaris@kustar.ac.ae}} \thanks{%
N. Elia is with the Department of Electrical Engineering, Iowa State
University, Ames, IA 50010 USA \texttt{\small nelia@iastate.edu}} \thanks{%
This work is partially supported by NSF awards CMMI-1663460, ECCS-1739732,
CCF-1717154, CIF-1220643, and AFOSR AF FA-9550-15-1-0119}}
\maketitle

\begin{abstract}
This paper deals with the problem of distributed control synthesis. We seek
to find structured controllers that are stably realizable over the
underlying network. We address the problem using an operator form of
discrete-time linear systems. This allows for uniform treatment of various
classes of linear systems, e.g., Linear Time Invariant (LTI), Linear Time
Varying (LTV), or linear switched systems. We combine this operator
representation for linear systems with the classical Youla parameterization
to characterize the set of stably realizable controllers for a given network
structure. Using this Youla Operator State-Space (YOSS) framework, we show
that if the structure satisfies certain subspace like assumptions, then both
stability and performance problems can be formulated as convex optimization
and more precisely as tractable model-matching problems to any a priori
accuracy. Furthermore, we show that the structured controllers found from
our approach can be stably realized over the network and provide a
generalized separation principle.
\end{abstract}

\section{Introduction}

Modern large-scale cyber-physical systems are composed of many
interconnected subsystems that are usually spread over a large geographic
area and communicate over a network. Many difficulties arise when designing
a centralized controller for such systems due to communication delays, the
structure of the underlying communication network, scalability, etc. Due to
these issues, there has been a shift towards designing decentralized
controllers, in which subcontrollers are designed and implemented for each
subsystem and they can communicate over the network.

Decentralized, structured and distributed controller design has attracted
the renewed attention of many researchers over the last 15 years or so.
Several new developments occurred using state space methods (e.g., in the
LMI framework \cite{dullerud2004distributed, langbort2004distributed}) which
suit quadratic criteria but could generally lead to suboptimal solutions. On
the other hand, input-output approaches using the Youla-parametrization were
found to be very powerful in providing truly optimal solutions for several
classes of structured problems by reducing them to convex problems over the
Youla parameter, encompassing a variety of criteria, including nonquadratic
(e.g., \cite{voulgaris2001convex, qi2004structured, voulgaris2003optimal,
BassamSCL05, rotkTAC06},\cite{sabuau2016convex},\cite{sabuau2014youla}). %

In the input-output, or transfer function, domain, the stabilizing
controllers are parametrized by the so-called Youla parameter, and the
search for the optimal Youla parameter is carried out over the space of
stable systems. Here, the order of the Youla parameter or that of the
controller is not assumed a priori. However, unlike the state-space
approaches, the realizability of the controller over the underlying
communication network may become an issue, if not taken directly into
account as pointed out in \cite{vamsi-cdc10, vamsi2010sub}. That is,
although the controller transfer function structure is compatible with the
underlying network communication graph, it may lead to an internally
unstable realization, i.e., a non-minimal realization with unstable pole
zero cancellations (e.g., \cite{vikas-tac10}, \cite{lessard2013structured}
and \cite{vamsi2016optimal}). Certain alternative input-output approaches
have recently been proposed (e.g., a system level approach in \cite%
{wang2017system} and references therein) that hold the potential to handle
certain optimal and stably realizable structured design, by convex
programming without resorting to Youla-parametrization. A potential drawback
is the need to solve an exact model-matching problem, i.e., equations that,
if possible to satisfy, may require infinite support of the LTI maps
involved and hence stopping criteria for the approximation by finite support
may not be precisely characterized.

In this paper, we propose a unified way to synthesize stably realizable
controllers with respect to any measure of performance, e.g., $l_{1}$, $%
l_{2} $, or $l_{\infty }$ induced norms. Our approach is based on utilizing
a state-space based operator form of the system and combining it with the
ideas in the Youla-parameterization. This has been developed initially in
the context of switching system analysis and design in \cite%
{naghnaeian2016unified, naghnaeian2016characterization}, and as it turns
out, it fits well for optimally solving structured problems \cite%
{naghnaeian2018unified}. Our approach involves revisiting classical robust
control problems, e.g. $\mathcal{H}_{\infty }$ or $l_{1}$, and proposing a
new way to check for stability and to parametrize the set of stabilizing
controllers. The stability check that we will develop in this paper is in
the form of a model-matching problem 
\begin{equation*}
\inf_{R\text{ stable.}}\left\Vert T_{1}+T_{2}RT_{3}\right\Vert ,
\end{equation*}%
where $T_{1}$, $T_{2}$, and $T_{3}$ are stable systems. Such problems are
convex and efficient algorithms exist to solve them with arbitrary accuracy.
As such, we depart from standard stability tests, e.g., eigenvalue or
quadratic Lyapunov results. This will allow us to propose a new way to
parameterize the set of all stabilizing controllers, which is especially
beneficial to designing decentralized controllers with stable realization.

Another novelty of this approach is a new separation principle. Although the
standard separation principle does not generally hold true in the context of
decentralized control, this newly formulated separation principle holds
valid. In order to present this new separation principle, first, we
introduce the class of full-information-like controllers. These are
controllers that map both the state and measured output to control input.
That is,%
\begin{equation}
u=\left[ 
\begin{array}{cc}
K_{1} & K_{2}%
\end{array}%
\right] \left[ 
\begin{array}{c}
x \\ 
y%
\end{array}%
\right] ,  \label{eq:500}
\end{equation}%
where $u$, $x$, and $y$ are control input, states, and measured output,
respectively; $K_{1}$ and $K_{2}$ are linear systems. The proposed
separation principle states that any full-information-like controller (\ref%
{eq:500}) yields an output feedback controller if $x$ is replaced by its
estimation $\hat{x}$, which in turn is generated by an observer from the
measured output. Furthermore, although the reader can focus on the LTI case
as a concrete example, these methods are general and hold for LTV, delayed,
and switching systems as well.

\section{Preliminaries}

In this paper, $\mathbb{R}$ and $\mathbb{Z}$ denote the sets of real numbers
and integers, respectively. The set of $n$-tuples $x=\left\{ x\left(
k\right) \right\} _{k=0}^{n-1}$ where $x\left( k\right) $s are real numbers
is denoted by $\mathbb{R}^{n}$. For any $x\in \mathbb{R}^{n}$, its $%
l_{\infty }$ and $l_{p}$ norms are defined as $\left\Vert x\right\Vert
_{\infty }=\max_{k\in \left\{ 0,1,...,n-1\right\} }\left\vert x\left(
k\right) \right\vert $ and $\left\Vert x\right\Vert _{p}=\left(
\sum_{k=0}^{n-1}\left\vert x\left( k\right) \right\vert ^{p}\right) ^{\frac{1%
}{p}}$, respectively. Let $g=\left\{ g\left( k\right) \right\}
_{k=0}^{\infty }$ be a sequence where $g\left( k\right) \in \mathbb{R}^{n}$.
Then, the $l_{\infty }$ and $l_{p}$ norm of this sequence are defined as $%
\left\Vert g\right\Vert _{\infty }=\sup_{k\in \mathbb{Z}_{+}}\left\Vert
g\left( k\right) \right\Vert _{\infty }$ and $\left\Vert g\right\Vert
_{p}=\left( \sum_{k=0}^{\infty }\left\Vert g\left( k\right) \right\Vert
_{p}^{p}\right) ^{\frac{1}{p}}$ whenever they are finite. The set of $%
\mathbb{R}^{n}$-valued sequences whose $l_{p}$ norm ($l_{\infty }$ norm) is
finite is denoted by $l_{p}^{n}$ ($l_{\infty }^{n}$). Given two normed
spaces $\left( X,\left\Vert .\right\Vert _{X}\right) $ and $\left(
Y,\left\Vert .\right\Vert _{Y}\right) $ and a linear operator $%
T:X\rightarrow Y$, its$\ $induced norm is defined as $\left\Vert
T\right\Vert _{X-Y}:=\sup_{f\neq 0}\frac{\left\Vert Tf\right\Vert _{Y}}{%
\left\Vert f\right\Vert _{X}}$. Whenever both vector spaces are $X$ , we use
the notation $\left\Vert T\right\Vert $ without any subscript if the result
holds for any induced norm. We will call $T$ bounded or stable if $%
||T||<\infty $.

Any linear causal map $T$, on the space of all sequences ($l_{\infty ,e}$),
can be thought of as an infinite dimensional lower triangular matrix, 
\begin{equation}
T=\left[ 
\begin{array}{ccc}
T_{0,0} & 0 & \cdots  \\ 
T_{1,1} & T_{1,0} & \cdots  \\ 
\vdots  & \vdots  & \ddots 
\end{array}%
\right] .  \label{eq:matrix}
\end{equation}%
Such a causal map $T$ is called \textit{strictly causal} if its diagonal
elements are zeros, i.e., $T_{0,0}=T_{1,0}=...=0$. We say $T$ is a \textit{%
square operator} if it has the same number of inputs as outputs, that is, $%
T_{i,j}$ terms are square matrices. Given a sequence $g=\left\{ g\left(
k\right) \right\} _{k=0}^{\infty }$, the delay or shift operator $\Lambda $
is defined by $\Lambda ^{k}g=\left\{ \underset{k\text{ zeros}}{\underbrace{%
0,...,0}},g\left( 0\right) ,g\left( 1\right) ,...\right\} $, and, with a
slight abuse of notation, $\Lambda ^{-k}g=\left\{ g\left( k\right) ,g\left(
k+1\right) ,...\right\} $. It is easy to show that if $T$ is a causal map
then $\Lambda T$ and $T\Lambda $ are strictly causal. Conversely, any
strictly causal operator $T$ can be written as $T=\Lambda \bar{T}$ where $%
\bar{T}$ is causal. A linear causal map $T$ is called time-invariant if it
commutes with the delay operator, i.e. $\Lambda T=T\Lambda $. If $T$ is a
Linear Time-Invariant (LTI), it is fully characterize by its impulse
response denoted by $\left\{ T\left( k\right) \right\} _{k=0}^{\infty }$. In
this case, its infinite dimensional matrix representation is given by%
\begin{equation*}
T=\left[ 
\begin{array}{ccc}
T\left( 0\right)  & 0 & \cdots  \\ 
T\left( 1\right)  & T\left( 0\right)  & \cdots  \\ 
\vdots  & \vdots  & \ddots 
\end{array}%
\right] .
\end{equation*}%
A finite dimensional LTI system has the state-space representation of 
\begin{equation}
G:\left\{ 
\begin{array}{c}
x\left( t+1\right) =Ax\left( t\right) +Bu\left( t\right)  \\ 
y\left( t\right) =Cx\left( t\right) +Du\left( t\right) 
\end{array}%
\right. ,\text{ with }x\left( t_{0}\right) =x_{0},  \label{eq:01'}
\end{equation}%
where $u\left( t\right) \in \mathbb{R}^{m},x\left( t\right) \in \mathbb{R}%
^{n},$ $y\left( t\right) \in \mathbb{R}^{p}$, and $x_{0}\in \mathbb{R}^{n}$
are input, state, output, and the initial condition of the system and $A$, $B
$, $C$, and $D$ are matrices with appropriate dimensions for all $t\in 
\mathbb{Z}_{+}$. Given a matrix $S$, we define $\hat{S}$ to be the diagonal
operator%
\begin{equation}
\hat{S}=\left[ 
\begin{array}{ccc}
S & 0 & \cdots  \\ 
0 & S &  \\ 
\vdots  &  & \ddots 
\end{array}%
\right] .  \label{eq:diag}
\end{equation}%
Using this notation, we can define diagonal operators $\hat{A}$, $\hat{B}$, $%
\hat{C}$, and $\hat{D}$ and rewrite (\ref{eq:01'}) as

\begin{equation}
G:\left\{ 
\begin{array}{c}
x=\Lambda \hat{A}x+\Lambda \hat{B}u+\bar{x}_{0} \\ 
y=\hat{C}x+\hat{D}u%
\end{array}%
\right. ,  \label{eq:01''}
\end{equation}%
where $\bar{x}_{0}=\left\{ \underset{t_{0}\text{ zeros}}{\underbrace{0,...,0}%
},x_{0},0,0,...\right\} $, $x=\left\{ x\left( t\right) \right\}
_{t=0}^{\infty }$, $y=\left\{ y\left( t\right) \right\} _{t=0}^{\infty }$, $%
u=\left\{ u\left( t\right) \right\} _{t=0}^{\infty }$, and $\Lambda $ is the
delay operator. The above representation of $G$ is referred to as a
realization in operator form. One can also write an operator realization for
the time delay systems. Consider the system given by%
\begin{equation}
H:\left\{ 
\begin{array}{c}
x\left( t+1\right) =\sum_{i=0}^{N}A_{i}x\left( t-i\right)
+\sum_{i=0}^{N}B_{i}u\left( t-i\right) \\ 
y\left( t\right) =\sum_{i=0}^{N}C_{i}x\left( t-i\right)
+\sum_{i=0}^{N}D_{i}u\left( t-i\right)%
\end{array}%
\right. ,  \label{eq:02'}
\end{equation}%
with initial condition $x_{0}=\left\{ x\left( k\right) \right\} _{k=-N}^{0}$%
. Define $\bar{A}=\sum_{i=0}^{N}\Lambda ^{i}\hat{A}_{i}$. Similarly, we
define $\bar{B},\bar{C}$, and $\bar{D}$. Then, the time-delay system can be
written in the operator form as%
\begin{equation}
H:\left\{ 
\begin{array}{c}
x=\Lambda \bar{A}x+\Lambda \bar{B}u+\bar{x}_{0} \\ 
y=\bar{C}x+\bar{D}u%
\end{array}%
\right. .  \label{eq:02}
\end{equation}%
In this operator framework, we make a distinction between an operator and
its realization as follows:

\begin{definition}[Operator Realization]
For a given linear (possibly unbounded) causal operator $T:u\rightarrow y$,
we will refer to the relationship%
\begin{equation*}
T:\left\{ 
\begin{array}{c}
x=\mathcal{A}_{T}x+\mathcal{B}_{T}u \\ 
y=\mathcal{C}_{T}x+\mathcal{D}_{T}u%
\end{array}%
\right. ,
\end{equation*}%
as an operator realization or simply a realization only if operators $%
\mathcal{A}_{T}$, $\mathcal{B}_{T}$, $\mathcal{C}_{T}$, and $\mathcal{D}_{T}$
are bounded causal operators and $\left( I-\mathcal{A}_{T}\right) ^{-1}$
exists.
\end{definition}

For example, realizations for LTI operators (\ref{eq:01'}) and delayed
systems (\ref{eq:02'}) are given in (\ref{eq:01''}) and (\ref{eq:02}),
respectively. Also, any bounded operator $T$ has a trivial realization with $%
\mathcal{A}_{T}=\mathcal{B}_{T}=\mathcal{C}_{T}=0$ and $\mathcal{D}_{T}=T$.
This realization, however, is not trivial for unstable operator $T$.

Throughout this paper, we prefer to write the systems in the operator form (%
\ref{eq:02}) as it allows for treating various classes of systems (e.g.
time-delay, switching, and LTV systems \cite{naghnaeian2016unified}) in a
unified way. Henceforth, we consider the systems that have operator forms as
in (\ref{eq:02}). Such a system can be seen as a mapping from $\bar{x}_{0}$
and $u$ to $x$ and $y$. For this system, we adopt the following definitions
of stability and gain.

\begin{definition}
\label{def:stab}Given two normed spaces $\left( \mathcal{U},\left\Vert
.\right\Vert _{\mathcal{U}}\right) $ and $\left( \mathcal{X},\left\Vert
.\right\Vert _{\mathcal{X}}\right) $, we say that the system $H$ in (\ref%
{eq:02}) is $\mathcal{U}$ to $\mathcal{X}$ stable if it is a bounded
operator from $\left[ \bar{x}_{0}^{T},u^{T}\right] ^{T}\in \mathcal{X\times U%
}$ to $\left[ x^{T},y^{T}\right] \in \mathcal{X\times X}$. More precisely, $%
H $ is $\mathcal{U}$ to $\mathcal{X}$ stable if, for some $\gamma
_{1},\gamma _{2}\geq 0$, $\left\Vert x\right\Vert _{\mathcal{X}}\leq \gamma
_{1}\left\Vert \bar{x}_{0}\right\Vert _{\mathcal{X}}+\gamma _{2}\left\Vert
u\right\Vert _{\mathcal{U}}$ and $\left\Vert y\right\Vert _{\mathcal{X}}\leq
\gamma _{1}\left\Vert \bar{x}_{0}\right\Vert _{\mathcal{X}}+\gamma
_{2}\left\Vert u\right\Vert _{\mathcal{U}}$ whenever $\left\Vert \bar{x}%
_{0}\right\Vert _{\mathcal{X}}$ and $\left\Vert u\right\Vert _{\mathcal{U}}$
are finite.
\end{definition}

\begin{definition}
\label{def:gain}Given two normed spaces $\left( \mathcal{U},\left\Vert
.\right\Vert _{\mathcal{U}}\right) $ and $\left( \mathcal{X},\left\Vert
.\right\Vert _{\mathcal{X}}\right) $, and a $\mathcal{U}$ to $\mathcal{X}$
stable system $H$, its gain is defined as $\left\Vert H\right\Vert _{%
\mathcal{U}-\mathcal{X}}=\sup_{\substack{ u\neq 0  \\ x_{0}=0}}\frac{%
\left\Vert y\right\Vert _{\mathcal{X}}}{\left\Vert u\right\Vert _{\mathcal{U}%
}}$.
\end{definition}

For simplicity, we let $\mathcal{X}$ and $\mathcal{U}$ to be the same (but
possibly with different dimension) $l_{p}$ spaces.

Finally, when need to ensure the invertibility of certain operators on $%
l_{\infty ,e}$, we will appeal to the following lemma:

\begin{lemma}
\label{lemma:inverse}The following hold:

\begin{enumerate}
\item Given a causal square operator $T$ as in (\ref{eq:matrix}), the
inverse $\left( I-T\right) ^{-1}$exists if $T_{i,0}$ is invertible for all $%
i=0,1,2,...$.

\item Given a partitioned squared operator $X=\left[ 
\begin{array}{cc}
X_{11} & X_{12} \\ 
X_{21} & X_{22}%
\end{array}%
\right] $, $\left( I-X\right) ^{-1}$ exists if $\left( I-X_{11}\right) $ and 
$\left( I-X_{22}\right) $ are invertible.
\end{enumerate}
\end{lemma}

\section{Basic Setup}

A standard practice for designing a distributed controller for subsystems
communicating over a given network is to aggregate all subsystems into one
system $P$ and design a controller for this system. The controller must be
designed in a way so that it can be implemented as subcontrollers
communicating over the given network. We consider the aggregate system with
a realization%
\begin{equation}
P:\left\{ 
\begin{array}{c}
x=\Lambda \bar{A}x+\Lambda \bar{B}_{1}w+\Lambda \bar{B}_{2}u+\bar{x}_{0} \\ 
z=\bar{C}_{1}x+\bar{D}_{11}w+\bar{D}_{12}u \\ 
y=\bar{C}_{2}x+\bar{D}_{12}w,%
\end{array}%
\right. ,  \label{eq:general}
\end{equation}%
where $x,y,$ and $z$ are the states, measurements, and the regulated output; 
$w$ and $u$ are the exogenous and control inputs; and, $\bar{A},\bar{B}_{i},%
\bar{C}_{j},\bar{D}_{ij}$, for $i.j\in \left\{ 1,2\right\} $ are bounded
operators.

\begin{example}
\label{ex:01}Consider a network with $N$ subsystems. Each subsystem is given
by%
\begin{eqnarray}
x_{i}\left( t+1\right) &=&A^{i}x_{i}\left( t\right) +B_{1}^{i}w_{i}\left(
t\right) +B_{2}^{i}u_{i}\left( t\right) +\sum_{j=1}^{N}B_{3}^{ij}\eta
_{ij}\left( t\right) ,  \notag \\
z_{i}\left( t\right) &=&C_{1}^{i}x_{i}\left( t\right) +D_{11}^{i}w_{i}\left(
t\right) +D_{12}^{i}u_{i}\left( t\right) ,  \notag \\
y_{i}\left( t\right) &=&C_{2}^{i}x_{i}\left( t\right) +D_{21}^{i}w_{i}\left(
t\right) ,  \notag \\
\nu _{ji}\left( t\right) &=&C_{3}^{ji}x_{i}\left( t\right)
+D_{31}^{ji}w_{i}\left( t\right) ,\text{ }j=1,2,...,N,  \label{eq:21}
\end{eqnarray}%
where $x_{i}$, $y_{i}$, and $z_{i}$ are the states, measured output, and
regulated output of the $i^{th}$ subsystem; $\nu _{ji}$ is the signal that
the $i^{th}$ subsystem communicates to the $j^{th}$ subsystem and $\eta
_{ij} $ is the signal that $i^{th}$ subsystem receives through its
communication link with the $j^{th}$ subsystem. We let $%
B_{2}^{ij}=C_{3}^{ji}=D_{31}^{ji}=0 $ if there is no communication link
between $i^{th}$ and $j^{th}$ subsystems. Furthermore, due to the delay in
the communication links, we set%
\begin{equation}
\eta _{ij}\left( t\right) =\nu _{ij}\left( t-\tau _{ij}\right) ,
\label{eq:20}
\end{equation}%
where $\tau _{ij}\in \mathbb{Z}_{\geq 0}$ is delay in communication from $%
j^{th}$ to $i^{th}$ subsystem. Substituting (\ref{eq:20}) in (\ref{eq:21}),
the $i^{th}$ subsystem, in the operator form, can be written as%
\begin{eqnarray*}
x_{i} &=&\Lambda \left[ \hat{A}^{i}x_{i}+\sum_{j=1}^{3}\Lambda ^{\tau _{ij}}%
\hat{A}^{ij}x_{j}\right] \\
&&+\Lambda \left[ \hat{B}_{1}^{i}w_{i}+\sum_{j=1}^{3}\Lambda ^{ij}\hat{B}%
_{1}^{ij}w_{j}\right] +\Lambda \hat{B}_{2}^{i}u_{i},
\end{eqnarray*}%
where $\hat{A}^{ij}=\hat{B}_{3}^{ij}\hat{C}_{3}^{ij}$ and $\hat{B}_{1}^{ij}=%
\hat{B}_{3}^{ij}\hat{D}_{31}^{ij}$. Based on the above expression, it can be
easily seen that for properly defined operators $\bar{A},\bar{B}_{i},\bar{C}%
_{i},\bar{D}_{ij}$, for $i\in \left\{ 1,2\right\} $, the aggregate system
can be written as in (\ref{eq:general}).
\end{example}

The structure of the network is reflected in the coefficient operators
involved in (\ref{eq:general}), e.g., as sparsity patterns \cite%
{vamsi2016optimal}. Given a fixed network consisting of $N$ nodes
(subsystems) and a set of $N$ inputs $\xi =[\xi _{i}]$ to, and $N$ outputs $%
\zeta =[\zeta _{i}]$ from, these $N$ nodes, let $\mathcal{S}$ denote the set
of all input-output maps (or, transfer functions in the LTI case) $T$ from $%
\zeta $ to $\xi $, i.e., $\xi =T\zeta $ that can be obtained from this
network. That is, the input-output aggregation of all subsystem (or,
subcontroller) dynamics, interconnected via the network, form an element $T$
in $\mathcal{S}$ and, conversely, any element in $\mathcal{S}$ can be
implemented, stably or unstably (to be precisely defined later), as
subsystems communicating over the given network. Consider the following
example:

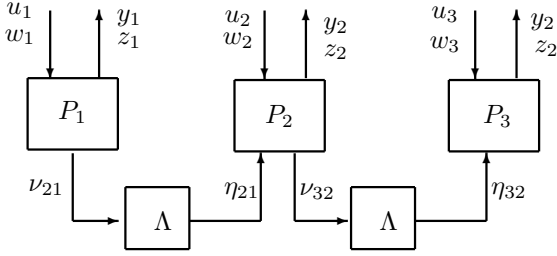
\begin{figure}[tbp] \centering%
\unitlength 0.8 mm
\begin{picture}(96.25,50)(0,40)
\linethickness{0.3mm}
\put(10,73.12){\line(1,0){15}}
\put(10,61.25){\line(0,1){11.88}}
\put(25,61.25){\line(0,1){11.88}}
\put(10,61.25){\line(1,0){15}}
\linethickness{0.3mm}
\put(44.38,72.81){\line(1,0){15}}
\put(44.38,60.94){\line(0,1){11.88}}
\put(59.38,60.94){\line(0,1){11.88}}
\put(44.38,60.94){\line(1,0){15}}
\linethickness{0.3mm}
\put(80,72.81){\line(1,0){15}}
\put(80,60.94){\line(0,1){11.87}}
\put(95,60.94){\line(0,1){11.87}}
\put(80,60.94){\line(1,0){15}}
\linethickness{0.3mm}
\put(13.75,73.12){\line(0,1){11.25}}
\put(13.75,73.12){\vector(0,-1){0.12}}
\linethickness{0.3mm}
\put(21.88,73.12){\line(0,1){11.25}}
\put(21.88,84.38){\vector(0,1){0.12}}
\put(17.5,67.5){\makebox(0,0)[cc]{$P_{1}$}}

\linethickness{0.3mm}
\put(17.5,49.38){\line(0,1){11.25}}
\linethickness{0.3mm}
\put(17.5,49.38){\line(1,0){7.5}}
\put(25,49.38){\vector(1,0){0.12}}
\linethickness{0.3mm}
\put(48.75,49.38){\line(0,1){11.25}}
\put(48.75,60.62){\vector(0,1){0.12}}
\linethickness{0.3mm}
\put(26.25,55){\line(1,0){10.62}}
\put(26.25,44.69){\line(0,1){10.31}}
\put(36.88,44.69){\line(0,1){10.31}}
\put(26.25,44.69){\line(1,0){10.62}}
\linethickness{0.3mm}
\put(36.88,49.38){\line(1,0){11.88}}
\put(51.88,66.88){\makebox(0,0)[cc]{$P_{2}$}}

\put(88.12,66.88){\makebox(0,0)[cc]{$P_{3}$}}

\linethickness{0.3mm}
\put(63.75,55){\line(1,0){10.62}}
\put(63.75,44.69){\line(0,1){10.31}}
\put(74.38,44.69){\line(0,1){10.31}}
\put(63.75,44.69){\line(1,0){10.62}}
\linethickness{0.3mm}
\put(54.38,49.38){\line(0,1){11.25}}
\linethickness{0.3mm}
\put(54.38,49.38){\line(1,0){8.75}}
\put(63.12,49.38){\vector(1,0){0.12}}
\linethickness{0.3mm}
\put(74.38,49.38){\line(1,0){11.88}}
\linethickness{0.3mm}
\put(86.25,49.38){\line(0,1){11.25}}
\put(86.25,60.62){\vector(0,1){0.12}}
\put(32.5,49.38){\makebox(0,0)[cc]{$\Lambda$}}

\put(70,49.38){\makebox(0,0)[cc]{$\Lambda$}}

\linethickness{0.3mm}
\put(49.38,73.12){\line(0,1){11.25}}
\put(49.38,73.12){\vector(0,-1){0.12}}
\linethickness{0.3mm}
\put(84.38,73.12){\line(0,1){11.25}}
\put(84.38,73.12){\vector(0,-1){0.12}}
\linethickness{0.3mm}
\put(56.25,73.12){\line(0,1){11.25}}
\put(56.25,84.38){\vector(0,1){0.12}}
\linethickness{0.3mm}
\put(91.25,73.12){\line(0,1){11.25}}
\put(91.25,84.38){\vector(0,1){0.12}}
\put(8.75,84.38){\makebox(0,0)[cc]{$u_{1}$}}

\put(45,83.12){\makebox(0,0)[cc]{$u_{2}$}}

\put(79.38,83.75){\makebox(0,0)[cc]{$u_{3}$}}

\put(26.88,83.12){\makebox(0,0)[cc]{$y_{1}$}}

\put(8.75,80){\makebox(0,0)[cc]{$w_{1}$}}

\put(26.88,79.38){\makebox(0,0)[cc]{$z_{1}$}}

\put(45,79.38){\makebox(0,0)[cc]{$w_{2}$}}

\put(79.38,78.75){\makebox(0,0)[cc]{$w_{3}$}}

\put(61.25,81.88){\makebox(0,0)[cc]{$y_{2}$}}

\put(95.62,82.5){\makebox(0,0)[cc]{$y_{2}$}}

\put(61.25,77.5){\makebox(0,0)[cc]{$z_{2}$}}

\put(96.25,78.12){\makebox(0,0)[cc]{$z_{2}$}}

\put(13.12,55){\makebox(0,0)[cc]{$\nu_{21}$}}

\put(45.62,54.38){\makebox(0,0)[cc]{$\eta_{21}$}}

\put(58.12,54.38){\makebox(0,0)[cc]{$\nu_{32}$}}

\put(90,54.38){\makebox(0,0)[cc]{$\eta_{32}$}}

\end{picture}
%
\caption{A simple nested network.}\label{fig:02}%
\end{figure}%

\begin{figure*}[tbp]
\centering\centering
\par
\includegraphics[trim= 0in 3.2in 0in 0.9in,clip,height=1.5in]{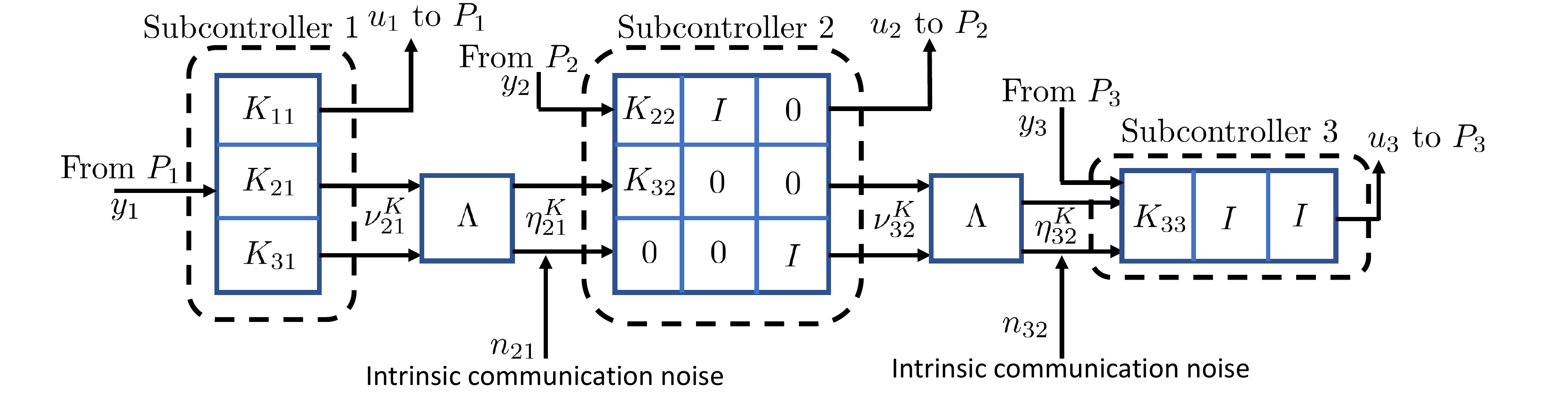}
\caption{Controller in Example \protect\ref{exam:03} implemented over the
network.}
\label{fig:03}
\end{figure*}

\begin{example}
\label{exam:02}\textbf{Nested network: }An example of a nested network is
given in Figure \ref{fig:02}. We adopt the notations introduced in Example %
\ref{ex:01}. It can be easily verified that the aggregate system is given by
(\ref{eq:general}) where $\bar{A}=\left\{ \bar{A}\left( 0\right) ,\bar{A}%
\left( 1\right) ,0,0,...\right\} $, $\bar{B}_{1}=\left\{ \bar{B}_{1}\left(
0\right) ,\bar{B}_{1}\left( 1\right) ,0,...\right\} $, $\bar{B}_{2}=\left\{ 
\bar{B}_{2}\left( 0\right) ,0,...\right\} $, $\bar{C}_{j}=\left\{ \bar{C}%
_{j}\left( 0\right) ,0,...\right\} $, $\bar{D}_{ij}=\left\{ \bar{D}%
_{ij}\left( 0\right) ,0,...\right\} $ with $\bar{A}\left( 0\right)
=diag\left\{ A^{1},A^{2},A^{3}\right\} $, $\bar{B}_{j}\left( 0\right)
=diag\left\{ B_{j}^{1},B_{j}^{2},B_{j}^{3}\right\} $, $\bar{C}_{j}\left(
0\right) =diag\left\{ C_{j}^{1},C_{j}^{2},C_{j}^{3}\right\} $, $\bar{D}%
_{ij}\left( 0\right) =diag\left\{ D_{ij}^{1},D_{ij}^{2},D_{ij}^{3}\right\} $%
, and for $i,j\in \left\{ 1,2\right\} $ 
\begin{eqnarray*}
\bar{A}\left( 1\right) &=&\left[ 
\begin{array}{ccc}
0 & 0 & 0 \\ 
B_{3}^{21}C_{3}^{21} & 0 & 0 \\ 
0 & B_{3}^{32}C_{3}^{32} & 0%
\end{array}%
\right] , \\
\bar{B}_{1}\left( 1\right) &=&\left[ 
\begin{array}{ccc}
0 & 0 & 0 \\ 
B_{3}^{21}D_{31}^{21} & 0 & 0 \\ 
0 & B_{3}^{32}D_{31}^{32} & 0%
\end{array}%
\right] .
\end{eqnarray*}%
In this example, the structure of the network is reflected on the impulse
response of the coefficient operators, e.g. $\bar{A}$. The terms in the
impulse response of $\bar{A}$ are lower triangular, which conforms with the
flow of communication from subsystem $1$ to subsystem $2$ and then to
subsystem $3$. And the sparsity structure in $\bar{A}\left( 0\right) $ and $%
\bar{A}\left( 1\right) $ is because each subsystem has immediate access to
its own measurement signal but communicates with its neighbors with a delay.
For this network, the set $\mathcal{S}$ is the space of all systems $P$
whose impulse response $\left\{ P\left( k\right) \right\} _{k=0}^{\infty }$
satisfies the following conditions: $P\left( k\right) $ is lower triangular
for $k=2,3,...$, $P\left( 0\right) $ is diagonal, and 
\begin{equation*}
P\left( 1\right) =\left[ 
\begin{array}{ccc}
\ast & 0 & 0 \\ 
\ast & \ast & 0 \\ 
0 & \ast & \ast%
\end{array}%
\right] ,
\end{equation*}%
where $\ast $ stands for a possibly non zero entry. Or, in transfer function
terms, 
\begin{equation*}
P\left[ \lambda \right] =\left[ 
\begin{array}{ccc}
\ast & 0 & 0 \\ 
\lambda \ast & \ast & 0 \\ 
\lambda ^{2}\ast & \lambda \ast & \ast%
\end{array}%
\right] ,
\end{equation*}%
where $P[\lambda ]=\sum_{k=0}^{\infty }\lambda ^{k}P(k)$ is the $\lambda $%
-transform. Accordingly, if $K$ is a controller for $P$ within the same
communication network, $K[\lambda ]$ should also be of the same form, i.e., $%
K\in \mathcal{S}$.
\end{example}

\bigskip

In the sequel, in generating the control input $u$, depending on what
information is available, we face two categories of problems,
full-information or output feedback. In full-information feedback, the
controller has access to the entire state and measured outputs, i.e., 
\begin{equation}
u=\left[ 
\begin{array}{cc}
K_{1} & K_{2}%
\end{array}%
\right] \left[ 
\begin{array}{c}
x \\ 
y%
\end{array}%
\right] .  \label{eq:FI}
\end{equation}%
In output feedback, however, $u$ must be generated only using the
information available in signal $y$, i.e. 
\begin{equation*}
u=\left[ 
\begin{array}{cc}
0 & K_{2}%
\end{array}%
\right] \left[ 
\begin{array}{c}
x \\ 
y%
\end{array}%
\right] .
\end{equation*}

\begin{remark}
We need to point out that there is a difference between the full-information
controller that we defined in (\ref{eq:FI}) and the one defined in \cite%
{zhou1996robust}. In the latter, the controller has access to the state $x$
and the disturbances $w$ (and consequently $y$) while the full-information
controller defined in (\ref{eq:FI}) only has access to $x$ and $y$. The two
definitions are equivalent if $\bar{D}_{21}$ is a square and invertible map.
\end{remark}

The generalized plant%
\begin{eqnarray}
\left[ 
\begin{array}{c}
x \\ 
y%
\end{array}%
\right] &=&\left[ 
\begin{array}{cc}
\Lambda \bar{A} & 0 \\ 
\bar{C}_{2} & 0%
\end{array}%
\right] \left[ 
\begin{array}{c}
x \\ 
y%
\end{array}%
\right] +\left[ 
\begin{array}{c}
\Lambda \bar{B}_{1}w+\Lambda \bar{B}_{2}u+\bar{x}_{0} \\ 
\bar{D}_{21}w%
\end{array}%
\right] ,  \notag \\
z &=&\left[ 
\begin{array}{cc}
\bar{C}_{1} & 0%
\end{array}%
\right] \left[ 
\begin{array}{c}
x \\ 
y%
\end{array}%
\right] +\bar{D}_{11}w+\bar{D}_{12}u,  \label{eq:plant_general}
\end{eqnarray}%
with the full-information feedback controller (or output feedback when $%
K_{1}=0$) results in the closed-loop system%
\begin{eqnarray}
&&\left[ 
\begin{array}{c}
x \\ 
y%
\end{array}%
\right] =\left[ 
\begin{array}{cc}
I-\Lambda \bar{A}-\Lambda \bar{B}_{2}K_{1} & -\Lambda \bar{B}_{2}K_{2} \\ 
-\bar{C}_{2} & I%
\end{array}%
\right] ^{-1}\times  \notag \\
&&\qquad \qquad \qquad \qquad \left\{ \left[ 
\begin{array}{c}
\Lambda \bar{B}_{1} \\ 
\bar{D}_{21}%
\end{array}%
\right] w+\left[ 
\begin{array}{c}
I \\ 
0%
\end{array}%
\right] \bar{x}_{0}\right\} ,  \notag \\
&&z=\left[ 
\begin{array}{cc}
\bar{C}_{1} & 0%
\end{array}%
\right] \left[ 
\begin{array}{c}
x \\ 
y%
\end{array}%
\right] +\bar{D}_{11}w+\bar{D}_{12}u,  \notag \\
&&u=\left[ 
\begin{array}{cc}
K_{1} & K_{2}%
\end{array}%
\right] \left[ 
\begin{array}{c}
x \\ 
y%
\end{array}%
\right] .  \label{eq:cl'}
\end{eqnarray}

This closed-loop system can be thought of as a linear operator from $\left( 
\begin{array}{c}
\bar{x}_{0} \\ 
w%
\end{array}%
\right) $ to signals $x$, $y$, $z$, and $u$. In conjunction with Definition %
\ref{def:stab}, we adopt the following definition for centralized
stabilizing controllers

\begin{definition}[Centralized Stabilizing Structured Controller]
\label{def:stab_central}A full-information or output feedback structured
controller $K=\left[ 
\begin{array}{cc}
K_{1} & K_{2}%
\end{array}%
\right] \in \mathcal{S}^{1\times 2}$ is said to be stabilizing in the
centralized way if the closed loop system (\ref{eq:cl'}) is a bounded
operator from $\bar{x}_{0}$ and $w$ to $x$, $y$, $z$, and $u$.
\end{definition}

It is important to note that an implementation of a controller in a
distributed way over a network is carried out through implementing a
realization of such controller that conforms with the network structure.
However, unless $K\in \mathcal{S}$ is a bounded operator itself, there is no
apriori guarantee that it has an operator realization conforming with the
structure of the network. In the next subsection, we will discuss this issue
further.

\subsection{Stable realizability/implementability}

The set $\mathcal{S}$ is fully characterized by the underlying network. In
this paper, given a (stabilizable and detectable in the usual sense)
generalized plant $P=\left[ 
\begin{array}{cc}
P_{zw} & P_{zu} \\ 
P_{yw} & P_{yu}%
\end{array}%
\right] $ as in (\ref{eq:general}), we are interested in finding the
controllers $K\in \mathcal{S}$ that are also stably realizable over the
network. We should point out that even if $K$ belongs to $\mathcal{S}$ and
stabilizes $P$ in the usual centralized sense (Definition \ref%
{def:stab_central}), i.e., if%
\begin{equation}
\left[ 
\begin{array}{cc}
I & P_{yu} \\ 
K & I%
\end{array}%
\right]  \label{central-stab}
\end{equation}%
has a stable inverse, it does not mean that the controller $K$ can
automatically be realized stably, although we can implement it as the
interconnection of subcontrollers consistent with the network. Unless we
guarantee that the implementation of $K$ does not have internal hidden
unstable modes, the closed-loop may not be stable. This is because a
stabilizing controller in a centralized sense, by design, guarantees the
boundedness (stability) of the measured output and control input, $y$ and $u$%
. Also, under the detectability assumption of the plant, the boundedness of $%
y$ and $u$ translates to that of $x$ and $z$ in (\ref{eq:general}). However,
when an aggregate controller is implemented in a decentralized way, new
signals are introduced and should taken into consideration. In particular,
these are the signals travelling between subcontrollers and also the
intrinsic noise on such signals. Therefore, in designing a controller that
is implementable over the network in a stable way, one needs to identify the
subcontrollers' communication signals and guarantee that the effects of
their intrinsic noise on every other signal of the system is bounded.

\begin{example}
\label{exam:03}Consider the nested network in Example \ref{exam:02} and let $%
K\in \mathcal{S}$ be a stabilizing controller in the centralized sense as in
(\ref{central-stab}). Since $K=\left\{ K\left( k\right) \right\}
_{k=0}^{\infty }\in \mathcal{S}$, it can be partitioned as $K=\left[ 
\begin{array}{ccc}
K_{11} & 0 & 0 \\ 
\Lambda K_{21} & K_{22} & 0 \\ 
\Lambda ^{2}K_{31} & \Lambda K_{32} & K_{33}%
\end{array}%
\right] $. One way to implement this controller over the network is
illustrated in Figure \ref{fig:03}. By definition, the signals $y$, $z$, and 
$u$ are bounded when $w$ and $\bar{x}_{0}$ are bounded. However, as
mentioned above, it is not guaranteed that the signals travelling between
the subcontrollers, i.e. $\nu _{21}^{K}$ and $\nu _{32}^{K}$, are bounded.
Furthermore, it is not clear if small bounded noise signals, shown as $n_{21}
$ and $n_{32}$ in the figure, corrupt the subcontrollers' communications,
the stability of the system, i.e., the boundedness other signals is
preserved. of Therefore, special attention should be paid to the stable
implementability of structured controllers.
\end{example}

Accordingly, we define the set $\mathcal{S}_{K}\subseteq \mathcal{S}$ to be
the set of stabilizing controllers $K$ that can be stably implemented over
the network without loosing stability, i.e., the subcontrollers communicate
bounded signals provided that the measured outputs, $y$, and control inputs, 
$u$, are bounded. We will make the definition of $\mathcal{S}_{K}$ concrete
in what follows.

\begin{definition}[Stably Realizable Controller ]
\label{def:stab_ctrl}Given a structured aggregate controller $K=\left[ 
\begin{array}{cc}
K_{1} & K_{2}%
\end{array}%
\right] \in \mathcal{S}^{1\times 2}$ ( $K_{1}=0$ for output feedback) that
is centralized stabilizing in the sense of Definition \ref{def:stab_central}%
, we say $K$ is stably realizable over the structure $\mathcal{S}$, i.e., $%
K\in \mathcal{S}_{K}$ if the following two conditions hold:

\begin{enumerate}
\item It has an operator realization 
\begin{equation}
K:\left\{ 
\begin{array}{c}
x_{K}=\mathcal{A}_{K}x_{K}+\mathcal{B}_{K}\bar{y} \\ 
u=\mathcal{C}_{K}x_{K}+\mathcal{D}_{K}\bar{y}%
\end{array}%
\right. ,  \label{eq:stable_real}
\end{equation}%
where

\begin{enumerate}
\item for full-information feedback $\bar{y}=\left[ x^{T},y^{T}\right] ^{T}$
and%
\begin{eqnarray*}
\mathcal{A}_{K} &\in &\mathcal{S}^{m\times m},\mathcal{B}_{K}\in \mathcal{S}%
^{m\times 2}, \\
\mathcal{C}_{K} &\in &\mathcal{S}^{1\times m},\mathcal{D}_{K}\in \mathcal{S}%
^{1\times 2},
\end{eqnarray*}

\item and for output feedback $\bar{y}=y,$%
\begin{eqnarray*}
\mathcal{A}_{K} &\in &\mathcal{S}^{m\times m},\mathcal{B}_{K}\in \mathcal{S}%
^{m\times 1}, \\
\mathcal{C}_{K} &\in &\mathcal{S}^{1\times m},\mathcal{D}_{K}\in \mathcal{S}%
^{1\times 1},
\end{eqnarray*}
\end{enumerate}

for some positive integer $m$ and stable operators $\mathcal{A}_{K}$, $%
\mathcal{B}_{K}$, $\mathcal{C}_{K}$, and $\mathcal{D}_{K}$.

\item The effects of fictitious subcontrollers' communication noise, denoted
by $n_{x}$ and $n_{u}$, on the system is bounded. That is, the
interconnection of "noisy controller"%
\begin{equation}
\bar{K}:\left\{ 
\begin{array}{c}
x_{K}=\mathcal{A}_{K}x_{K}+\mathcal{B}_{K}\bar{y}+n_{x} \\ 
u=\mathcal{C}_{K}x_{K}+\mathcal{D}_{K}\bar{y}+n_{u}%
\end{array}%
\right. ,  \label{eq:noisy_ctrl}
\end{equation}%
and the plant $P$ in (\ref{eq:general}) results in a bounded closed-loop
operator from $\bar{x}_{0}$, $w$, $n_{x}$, and $n_{y}$ to $x$, $y$, $x_{K}$,
and $u$.

We refer to (\ref{eq:stable_real}) as a stable realization or implementation
of the controller $K$.
\end{enumerate}
\end{definition}

\bigskip

Given $\mathcal{S}$, it is not clear that every centralized stabilizing
controller (in the sense of Definition \ref{def:stab_central}), $K\in 
\mathcal{S}$, can be stably realized over the network structure. Therefore,
we always have $\mathcal{S}_{K}\subseteq \mathcal{S}$. In this paper, given
a structure $\mathcal{S}$, we develop necessary and sufficient conditions in
terms of convex problems under which it is possible to find optimal
structured controllers $K\in \mathcal{S}_{K}$ that are stably implementable
over the network. A typical $S$ of interest consists of controllers with
certain sparsity or delay patterns. Furthermore, we assume that $\mathcal{S}$
is a subspace that satisfies the following:

\begin{assumption}
\label{assumption1}The set $\mathcal{S}$ is delay-invariant, contains
identity and zero, and is closed under addition and multiplication, i.e., $%
I,0\in \mathcal{S}$, $\Lambda \mathcal{S\in S}$ and for any $X,Y\in \mathcal{%
S}$, $X+Y\in \mathcal{S}$ and $XY\in \mathcal{S}.$
\end{assumption}

\begin{assumption}
\label{assumption2}The set $\mathcal{S}$ contains the coefficient operator $%
\bar{A}$, $\bar{B}_{2}$, and $\bar{C}_{2}$.
\end{assumption}

Under such assumptions, as we will show $\mathcal{S}_{K}=\mathcal{S}$. In
addition to the above assumption, we make the following assumption
pertaining to the fact that all of the measured outputs, $y$, are corrupted
by noise.

\begin{assumption}
\label{assumption3}$\bar{D}_{21}$ has a trivial left null space.
\end{assumption}

\section{Youla Operator State-Space Parametrization of Stabilizing
Controllers}

In this section, we revisit a classical robust result on parametrizing the
set of stabilizing controllers. Traditionally, a typical approach to find
the set of stabilizing controllers is via Youla-Kucera parametrization that
utilizes the doubly coprime factorization. Here, we propose a new approach
in the operator framework that do not require coprime factorization. This
approach is referred to as the \textit{Youla Operator State-Space} (YOSS)
and proves to be particularly powerful in designing decentralized
controllers for linear systems.

\subsection{YOSS for Full-Information Feedback}

For the full-information feedback problems, we first parametrize the set of
all centralized stabilizing structured controllers in the sense of
Definition \ref{def:stab_central}. And then, we will show how such
controllers can be stably realized/implemented in a distributed way over the
network structure in the sense of Definition \ref{def:stab_ctrl}. In order
to state the result, we define an affine expression $\mathcal{E}_{Q,Z}$ in
terms of bounded operators $Q$ and $Z$ as follows:%
\begin{equation}
\mathcal{E}_{Q,Z}:=\left[ 
\begin{array}{cc}
\Lambda \bar{A} & 0 \\ 
\bar{C}_{2} & 0%
\end{array}%
\right] +\left[ 
\begin{array}{cc}
\Lambda \bar{A}-I & 0 \\ 
\bar{C}_{2} & -I%
\end{array}%
\right] Q+\left[ 
\begin{array}{c}
\Lambda \bar{B}_{2} \\ 
0%
\end{array}%
\right] Z.  \label{eq:epsilon}
\end{equation}%
Then, the following theorem holds:

\begin{theorem}
\label{thm:state-feedback}The following conditions are equivalent:

\begin{enumerate}
\item There exists a centralized stabilizing structured full-information
feedback controller $K\in \mathcal{S}:\left[ x^{T},y^{T}\right]
^{T}\rightarrow u$ for the the plant (\ref{eq:plant_general}).

\item There exist stable causal operators $Q$ and $Z$ and \textbf{some} $%
\varepsilon \in \lbrack 0,1)$ such that:

\textbf{a) Model-Matching:}%
\begin{equation}
\left\Vert \mathcal{E}_{Q,Z}\right\Vert \leq \varepsilon .  \label{cond1}
\end{equation}

\textbf{b) Structure:} $Q$ and $Z$ can be partitioned as%
\begin{eqnarray}
Q &=&\left[ 
\begin{array}{cc}
\Lambda Q_{11} & \Lambda Q_{12} \\ 
Q_{21} & \Lambda Q_{22}%
\end{array}%
\right] ,Z=\left[ 
\begin{array}{cc}
Z_{1} & Z_{2}%
\end{array}%
\right] ,  \notag \\
Q_{i,j} &\in &\mathcal{S},Z_{i}\in \mathcal{S}\text{ for }i,j=1,2.
\label{eq:QZ}
\end{eqnarray}

\item For \textbf{any} $\varepsilon \in \lbrack 0,1)$, there exist stable
causal operators $Q^{\varepsilon }$ and $Z^{\varepsilon }$ such that
Conditions 2.a and 2.b hold for $Q=Q^{\varepsilon }$ and $Z=Z^{\varepsilon }$%
.

\item There exist stable causal operators $Q^{0}$ and $Z^{0}$ such that
Conditions 2.a and 2.b hold for $\varepsilon =0$, $Q=Q^{0}$, and $Z=Z^{0}$.
\end{enumerate}
\end{theorem}

\begin{corollary}
\label{cor:01}Given a centralized full-information feedback controller $K\in 
\mathcal{S}^{1\times 2}$, there exists stable operators $Q$ and $Z$ such
that (\ref{cond1}) and (\ref{eq:QZ}) hold for some $\varepsilon \in \lbrack
0,1)$ and $K$ can be decomposed as%
\begin{equation}
K=\left[ 
\begin{array}{cc}
K_{1} & K_{2}%
\end{array}%
\right] =Z\left( I+Q\right) ^{-1},  \label{eq:400}
\end{equation}%
where the inverse $\left( I+Q\right) ^{-1}$ exists by Lemma \ref%
{lemma:inverse}.
\end{corollary}

Theorem \ref{thm:state-feedback} provides a convex way to parametrize the
set of all centralized stabilizing controllers that have the structure.
Furthermore, Corollary \ref{cor:01} provides a factorization for the
controller that can be used to obtain an operator realization for the
controller. To this end, note that according to (\ref{eq:400})%
\begin{equation*}
u=Z\left( I+Q\right) ^{-1}\left[ x^{T},y^{T}\right] ^{T}.
\end{equation*}%
Defining 
\begin{equation*}
x_{K}=\left( I+Q\right) ^{-1}\left[ x^{T},y^{T}\right] ^{T},
\end{equation*}%
a full-information controller can be realized as%
\begin{equation}
K:\left\{ 
\begin{array}{c}
x_{K}=\mathcal{A}_{K}x_{K}+\mathcal{B}_{K}\bar{y} \\ 
u=\mathcal{C}_{K}x_{K}+\mathcal{D}_{K}\bar{y}%
\end{array}%
\right. ,  \label{eq:FI_real}
\end{equation}%
where $\bar{y}=\left[ x^{T},y^{T}\right] ^{T}$ and%
\begin{eqnarray*}
\mathcal{A}_{K} &=&-Q\in \mathcal{S}^{2\times 2},\mathcal{B}_{K}=I_{2\times
2}\in \mathcal{S}^{2\times 2}, \\
\mathcal{C}_{K} &=&Z\in \mathcal{S}^{1\times 2},\mathcal{D}_{K}=0_{1\times
2}\in \mathcal{S}^{1\times 2},
\end{eqnarray*}%
with $Q\ $and $Z$ satisfying (\ref{cond1}) for some $\varepsilon \in \lbrack
0,1)$. This is an operator realization of the controller that conforms with
the network structure. In order to show that this realization provides a
stable implementation, in the sense of Definition \ref{def:stab_ctrl}, we
need to show that the noisy version of such a controller with fictitious
subcontrollers' communication noise as given in (\ref{eq:noisy_ctrl})
results in a bounded closed-loop map.

\begin{theorem}
\label{thm:FI_real}Any centralized stabilizing full-information controller
with structure can be stably realized over the network structure through
realization (\ref{eq:FI_real}). That is, $\mathcal{S}_{K}=\mathcal{S}$.
\end{theorem}

According to Theorem \ref{thm:state-feedback}, parametrizing the set of all
full-information stabilizing controllers amounts to solving an optimization
problem (\ref{cond1}) over the space of stable operators $\left( Q,Z\right) $
for some given $\varepsilon \in \lbrack 0,1)$. Such optimization problems
are in the generic form $\inf_{R}\left\Vert T_{1}+T_{2}RT_{3}\right\Vert $,
where $T_{1}$, $T_{2}$, $T_{3}$, and $R$ are stable operators. This
optimization, which is commonly referred to as a \textit{model-matching
problem} in robust control community, is convex and there are efficient
methods to solve it with arbitrary accuracy. In particular, when the
operator norm is taken to be the $l_{\infty }$ induced norm, this
optimization can be cast as a linear program.

\begin{remark}
According to Theorem \ref{thm:state-feedback}, although the model-matching (%
\ref{cond1}) can be carried out with $\varepsilon =0$, it is advantageous,
for computations, to relax the condition to $\varepsilon \in \lbrack 0,1)$
as stated in the theorem. As a matter fact, if $\varepsilon $ is set to
zero, there might not exist any finite-impulse response $Q$ and $Z$ that
satisfy (\ref{cond1}). For example, consider a trivial uncontrollable but
stable system%
\begin{equation*}
x\left( t+1\right) =\frac{1}{2}x\left( t\right) +w\left( t\right) +0u\left(
t\right) ,y\left( t\right) =x\left( t\right) +w\left( t\right) .
\end{equation*}%
Then, it can be easily verified that $Q=\left[ 
\begin{array}{cc}
\Lambda Q_{11} & 0 \\ 
I+\Lambda Q_{11} & 0%
\end{array}%
\right] $, and any stable $Z$ satisfy (\ref{cond1}) for $\varepsilon =0$ if
and only if $Q_{11}$ has the impulse response $\left\{ \frac{1}{2},\left( 
\frac{1}{2}\right) ^{2},\left( \frac{1}{2}\right) ^{3},...\right\} $. This
is particularly important since the approximation with FIR will fail to
provide performance guarantees.
\end{remark}

\subsection{YOSS for Output Feedback}

The output feedback controllers form a subset of full-information feedback
controllers, which were parametrized in Theorem \ref{thm:state-feedback} and
Corollary \ref{cor:01}. More precisely, a full-information controller (\ref%
{eq:400}) is an output feedback controller if $K_{1}=0$ or equivalently%
\begin{equation}
Z\left( I+Q\right) ^{-1}\left[ 
\begin{array}{c}
I \\ 
0%
\end{array}%
\right] =0.  \label{cond2_output}
\end{equation}%
Therefore, we have the following theorem.

\begin{theorem}
\label{lem:100}The following conditions are equivalent:

\begin{enumerate}
\item There exists a centralized stabilizing output-feedback controller $%
K:y\rightarrow u$.

\item There exist stable causal operators $Q$ and $Z$ and \textbf{some} $%
\varepsilon \in \lbrack 0,1)$ such that (\ref{cond1})-(\ref{eq:QZ}) and (\ref%
{cond2_output}) hold.

\item For \textbf{any} $\varepsilon \in \lbrack 0,1)$, there exist stable
causal operators $Q^{\varepsilon }$ and $Z^{\varepsilon }$ such that
Condition 2 holds for $Q=Q^{\varepsilon }$ and $Z=Z^{\varepsilon }$.

\item There exist stable causal operators $Q^{0}$ and $Z^{0}$ such that (\ref%
{cond1}), (\ref{eq:QZ}), and (\ref{cond2_output}) hold for $Q=Q^{0}$, $%
Z=Z^{0}$, and $\varepsilon =0$.
\end{enumerate}

If any of the above equivalent conditions hold, a stabilizing
output-feedback controller $K:y\rightarrow u$ is given by%
\begin{equation*}
K=Z\left( I+Q\right) ^{-1}\left[ 
\begin{array}{c}
0 \\ 
I%
\end{array}%
\right] .
\end{equation*}
\end{theorem}

\bigskip

Although finding the set of stabilizing full-information feedback
controllers, as given in Theorem \ref{thm:state-feedback}, is a convex and
tractable problem, this is not the case for output feedback due to the
non-convexity associated with (\ref{cond2_output}). We emphasize that (\ref%
{cond2_output}) is enforced to restrict the set of full-information
stabilizing controllers, $u=\left[ 
\begin{array}{cc}
K_{1} & K_{2}%
\end{array}%
\right] \left[ 
\begin{array}{c}
x \\ 
y%
\end{array}%
\right] $, to the ones whose $K_{1}$ element is zero. In what follows, we
will relax condition (\ref{cond2_output}), in order to preserve the
convexity, and will replace $x$ with its estimation $\hat{x}$, which is
generated based on $y$, in order to preserve the output feedback structure.
More precisely, we will consider controllers of the following form%
\begin{equation*}
u=\left[ 
\begin{array}{cc}
K_{1} & K_{2}%
\end{array}%
\right] \left[ 
\begin{array}{c}
\hat{x} \\ 
y%
\end{array}%
\right] ,
\end{equation*}%
where $\hat{x}$ is an estimation of $x$ generated from observing $y$ and $u$
such that the estimation error, $e=\hat{x}-x$, is a bounded signal. The estimation $\hat{x}$ is generated by an
state-estimator in the following generic form%
\begin{equation}
\hat{x}=E_{1}u+E_{2}y,  \label{eq:estimation}
\end{equation}%
where $E_{1}$ and $E_{2}$ are causal operators. Later (in Theorem \ref%
{thm:SE}), we will show that there always exist bounded operators $E_{1}$
and $E_{2}$ conforming with the network structure, i.e., $E_{1},E_{2}\in 
\mathcal{S}$.

\begin{theorem}
\label{thm:cond_output}Suppose $\hat{x}$ is a an estimation of $x$ generated
via the generic estimator (\ref{eq:estimation}) such that $E_{1},E_{2}\in 
\mathcal{S}$ and $\left\Vert e\right\Vert \leq \delta $ for some $\delta >0$%
, where $e=\hat{x}-x$. Then, the following conditions are equivalent:

\begin{enumerate}
\item A structured output-feedback controller $\tilde{K}\mathcal{\in S}%
:y\rightarrow u$ is stabilizing when implemented in the centralized way.

\item There exists causal operators $K_{1},K_{2}\in \mathcal{S}$ such that%
\begin{equation}
u=\left[ 
\begin{array}{cc}
K_{1} & K_{2}%
\end{array}%
\right] \left[ 
\begin{array}{c}
\hat{x} \\ 
y%
\end{array}%
\right] =\tilde{K}y.  \label{eq:ctrl0}
\end{equation}

\item There exist stable causal operators $Q$ and $Z$ and \textbf{some} $%
\varepsilon \in \lbrack 0,1)$ such that%
\begin{equation}
\left[ 
\begin{array}{cc}
K_{1} & K_{2}%
\end{array}%
\right] =Z\left( I+Q\right) ^{-1},  \label{eq:ctrl}
\end{equation}%
and $Q$ and $Z$ satisfy:

\begin{enumerate}
\item Model-Matching (\ref{cond1}) holds.

\item Structure constraints (\ref{eq:QZ}) are satisfied.
\end{enumerate}

\item For \textbf{any} $\varepsilon \in \lbrack 0,1)$, there exist stable
causal operators $Q^{\varepsilon }$ and $Z^{\varepsilon }$ such that
Condition 3 holds for $Q=Q^{\varepsilon }$ and $Z=Z^{\varepsilon }$.
\end{enumerate}
\end{theorem}

In the light of this theorem, the parametrization of all stabilizing output
feedback controllers can be carried out via a convex search over parameters $%
\left( Q,Z\right) $ provided that an estimation $\hat{x}$ of the states $x$
with uniform bound, $\left\Vert \hat{x}-x\right\Vert \leq \delta $, is
available, which is the subject of next section. By uniformity we mean that
the bound $\delta $ does not depend on the control input $u$ and is uniform
with respect to $w$ and $\bar{x}_{0}$. More precisely, the should exist $%
\delta _{1},\delta _{2}\geq 0$ such that $\left\Vert e\right\Vert \leq
\delta _{1}\left\Vert \bar{x}_{0}\right\Vert +\delta _{2}\left\Vert
w\right\Vert $. The parametrization presented here does not require the
doubly coprime factorization of the plant and this is an advantage
especially when designing decentralized controllers. Furthermore, this
parameterization holds valid not only for LTI systems but also any other
linear system, e.g., delayed , time-varying, or switching systems.

\bigskip

\subsubsection{State-Estimator}

Here, we will define an state-estimator that mimics the standard Luenberger
observer with a difference that its observer gain is a possibly unbounded
operator. This state-estimator is referred to as the \textit{Generalized
Luenberger Observer }and has the form%
\begin{equation}
\hat{x}=\Lambda \bar{A}\hat{x}+\Lambda \bar{B}_{2}u+\Lambda L\left( \bar{C}%
_{2}\hat{x}-y\right) ,  \label{LO}
\end{equation}%
where $L$ is the observer operator-gain which can possibly be unstable.

\begin{theorem}
\label{thm:SE}If there exists a centralized stabilizing output feedback
controller with structure, there always exists a generalized Luenberger
observer (\ref{LO}) with the network structure. The observer operator-gain
is given by%
\begin{equation}
L=\left( I+Q_{L}^{\varepsilon }\Lambda \right) ^{-1}Z_{L}^{\varepsilon }\in 
\mathcal{S}  \label{OG}
\end{equation}%
where, for any $\varepsilon \in \lbrack 0,1)$, $Q_{L}^{\varepsilon }\in 
\mathcal{S}$ and $Z_{L}^{\varepsilon }\in \mathcal{S}$ are stable causal
operators satisfying $\left\Vert \mathcal{E}_{L}\right\Vert \leq \varepsilon 
$, where 
\begin{equation}
\mathcal{E}_{L}:=\bar{A}+Z_{L}^{\varepsilon }\bar{C}_{2}-Q_{L}^{\varepsilon
}\left( I-\Lambda \bar{A}\right) .  \label{eq:606}
\end{equation}%
In this case, the state-estimator is simplified to 
\begin{equation}
\hat{x}=R_{1}\Lambda \bar{B}_{2}u-R_{2}y,  \label{eq:605}
\end{equation}%
where 
\begin{equation}
R_{1}:=\left[ I-\Lambda \left( \bar{A}+L\bar{C}_{2}\right) \right] ^{-1}\in 
\mathcal{S}{\normalsize ,}  \label{eq:R1}
\end{equation}%
and 
\begin{equation}
R_{2}:=R_{1}\Lambda L\in \mathcal{S},  \label{eq:R2}
\end{equation}%
are bounded operators. Furthermore, the estimation error $e$ is given by%
\begin{equation}
e=\hat{x}-x=-\left( R_{1}\Lambda \bar{B}_{1}+R_{2}\bar{D}_{21}\right) w-R_{1}%
\bar{x}_{0}.  \label{eq:error}
\end{equation}
\end{theorem}

A realization of the state-estimator is given by (\ref{eq:606}) and can also
be rewritten as%
\begin{equation}
\hat{x}=\Lambda \mathcal{E}_{L}\hat{x}+\left( I+\Lambda Q_{L}^{\varepsilon
}\right) \Lambda \bar{B}_{2}u-\Lambda Z_{L}^{\varepsilon }y.  \label{SE}
\end{equation}

\subsubsection{Stable Realization of Output Feedback}

The operator forms (\ref{eq:ctrl}) and (\ref{SE}) provide basis for stable
realization of the overall controller over the network. The controller (\ref%
{eq:ctrl}) can be realized as%
\begin{eqnarray}
\left[ 
\begin{array}{c}
\xi _{1} \\ 
\xi _{2}%
\end{array}%
\right] &=&-\left[ 
\begin{array}{cc}
\Lambda Q_{11} & \Lambda Q_{12} \\ 
Q_{21} & \Lambda Q_{22}%
\end{array}%
\right] \left[ 
\begin{array}{c}
\xi _{1} \\ 
\xi _{2}%
\end{array}%
\right] +\left[ 
\begin{array}{c}
\hat{x} \\ 
y%
\end{array}%
\right] ,  \notag \\
u &=&\left[ 
\begin{array}{cc}
Z_{1} & Z_{2}%
\end{array}%
\right] \left[ 
\begin{array}{c}
\xi _{1} \\ 
\xi _{2}%
\end{array}%
\right] .  \label{eq:re_ctrl0}
\end{eqnarray}%
Then, combining the realization of the state-estimator (\ref{SE}) with (\ref%
{eq:re_ctrl0}), we obtain the following realization for the output feedback
controller $K:y\rightarrow u$:%
\begin{equation}
K:\left\{ 
\begin{array}{c}
x_{K}=\mathcal{A}_{K}x_{K}+\mathcal{B}_{K}y \\ 
u=\mathcal{C}_{K}x_{K}%
\end{array}%
\right. ,  \label{eq:realization}
\end{equation}%
where%
\begin{eqnarray}
&&\mathcal{A}_{K}=\left[ 
\begin{array}{ccc}
\Lambda \mathcal{E}_{L} & \left( I+\Lambda Q_{L}^{\varepsilon }\right)
\Lambda \bar{B}_{2}Z_{1} & \left( I+\Lambda Q_{L}^{\varepsilon }\right)
\Lambda \bar{B}_{2}Z_{2} \\ 
I & -\Lambda Q_{11} & -\Lambda Q_{12} \\ 
0 & -Q_{21} & -\Lambda Q_{22}%
\end{array}%
\right] ,  \notag \\
&&\mathcal{B}_{K}=\left[ 
\begin{array}{c}
-\Lambda Z_{L}^{\varepsilon } \\ 
0 \\ 
I%
\end{array}%
\right] ,\mathcal{C}_{K}=\left[ 
\begin{array}{ccc}
0 & Z_{1} & Z_{2}%
\end{array}%
\right] .  \label{eq:8000}
\end{eqnarray}%
In the above expression, $x_{K}$ is given by $x_{K}=\left[ \hat{x}^{T},\xi
_{1}^{T},\xi _{2}^{T}\right] ^{T}$, where $\hat{x}$ is the state-estimator
as given in (\ref{SE}) and $\xi _{1}$ and $\xi _{2}$ are the signals given
in (\ref{eq:re_ctrl0}). The realization (\ref{eq:realization}) is
implementable over the network as operators $\mathcal{A}_{K}$, $\mathcal{B}%
_{K}$, and $\mathcal{C}_{K}$ conform with the network structure $\mathcal{S}$
\footnote{%
For sparsity structures, one can find the realization of each subcontroller,
from the realization of the aggregate controller, using a procedure similar
to that of \cite{sabuau2014youla}.}. We, however, need to show that such
implementation is stable. That is, the stability of the system is preserved
even if the signals traveling between the subcontrollers are subject to
noise, that is, if the noisy controller 
\begin{equation}
\bar{K}:\left\{ 
\begin{array}{c}
x_{K}=\mathcal{A}_{K}x_{K}+\mathcal{B}_{K}y+n_{x} \\ 
u=\mathcal{C}_{K}x_{K}+n_{u}%
\end{array}%
\right.  \label{eq:real_n}
\end{equation}%
is implemented. In other words, we need to show that the closed-loop maps
from subcontrollers communication noise $n_{x}$ and $n_{u}$ to signals $x$, $%
y$, $u$, and $x_{K}$ are bounded.

\begin{theorem}
\label{thm:realization}The realization given in (\ref{eq:realization}) can
be implemented stably over the given network structure $\mathcal{S}$. That
is, the feedback interconnection of $\bar{K}$ and the plant results in
bounded signals $x$, $y$, $x_{K}$ and $u$. This is , $\mathcal{S}_{K}=%
\mathcal{S}$.
\end{theorem}

\section{Optimal Control Synthesis}

In the last section, we parametrized the set of all structured stabilizing
controllers via a convex model-matching problem (\ref{cond1}) while
enforcing the network structure on components of $Q$ and $Z$ as in (\ref%
{eq:QZ}) . These conditions are identical for full-information and output
feedback controllers. Furthermore, a stable realization for stabilizing
controllers was given as in (\ref{eq:FI_real}) for full-information feedback
and (\ref{eq:realization}) for output feedback controller. This subsection
is devoted to finding the optimal controller such that the closed-loop gain
from the exogenous input $w$ to regulated output $z$ is minimized. To this
end, first, we will derive the set of all closed-loop maps from $w$ to $z$
when an output stabilizing controller is utilized. The set of all output
stabilizing controllers is parametrized by Theorem \ref{thm:cond_output}.
For fixed $\varepsilon \in \lbrack 0,1)$, the set of all such controllers
are given by%
\begin{equation}
u=Z\left( I+Q\right) ^{-1}\left[ 
\begin{array}{c}
\hat{x} \\ 
y%
\end{array}%
\right] ,  \label{eq:ctrl1}
\end{equation}%
where $\left( Q,Z\right) $ satisfies (\ref{cond1})-(\ref{eq:QZ}) and $\hat{x}
$ is an state estimation. First, we find and fix a state-estimator (\ref{SE}%
) with structure as in Theorem \ref{thm:SE}. The the following holds:

\begin{proposition}
\label{prop:1}The generalized plant (\ref{eq:plant_general}) with the output
feedback controller (\ref{eq:ctrl1}) results in the following closed-loop
map $\Phi _{wz}$ from $w$ to $z$:%
\begin{equation}
\Phi _{wz}=H+U\left[ 
\begin{array}{c}
Q \\ 
Z%
\end{array}%
\right] V+U\left[ 
\begin{array}{c}
I+Q \\ 
Z%
\end{array}%
\right] \mathcal{E}\left( I-\mathcal{E}\right) ^{-1}V,
\label{eq:closed_maps}
\end{equation}%
where $\mathcal{E=E}_{Q,Z}$ given in (\ref{eq:epsilon}),%
\begin{eqnarray}
H &=&\bar{C}_{1}\Lambda \bar{B}_{1}+\bar{C}_{1}\Lambda \bar{A}\left(
R_{1}\Lambda \bar{B}_{1}+R_{2}\bar{D}_{21}\right) +\bar{D}_{11},
\label{eq:700} \\
U &=&\left[ 
\begin{array}{cc}
\left[ 
\begin{array}{cc}
\bar{C}_{1} & 0%
\end{array}%
\right] & \bar{D}_{12}%
\end{array}%
\right] ,  \label{eq:701} \\
V &=&\left\{ \left[ 
\begin{array}{c}
\Lambda \bar{B}_{1} \\ 
\bar{D}_{21}%
\end{array}%
\right] +\left[ 
\begin{array}{c}
\Lambda \bar{A}-I \\ 
\bar{C}_{2}%
\end{array}%
\right] \left( R_{1}\Lambda \bar{B}_{1}+R_{2}\bar{D}_{21}\right) \right\} ,
\label{eq:702}
\end{eqnarray}%
with $R_{1}$ and $R_{2}$ being stable operators defined in (\ref{eq:R1})-(%
\ref{eq:R2}).
\end{proposition}

Notice that $\mathcal{E}$ as defined in (\ref{eq:epsilon}) can be made
arbitrarily small according to Theorem \ref{thm:cond_output}. Therefore,
without loss of generality, one can set $\mathcal{E}$ in (\ref{eq:epsilon})
equal to zero; in which case, the closed-loop map $\Phi _{wz}$ will be an
affine function of stable operators $Q$ and $Z$. This is stated in the
following theorem without any further proof:

\begin{theorem}
Given a fixed generalized Luenberger observer, as provided in Theorem \ref%
{thm:SE}, any stably realizable output feedback controller can be written as
a mapping from state-estimator $\hat{x}$ and measured output $y$ to control
input $u$, i.e., $K:\left[ 
\begin{array}{c}
\hat{x} \\ 
y%
\end{array}%
\right] \rightarrow u$, in the form (\ref{eq:ctrl1}),where $\left(
Q,Z\right) $ satisfies (\ref{cond1})-(\ref{eq:QZ}) for $\varepsilon =0$.
Then, the closed-loop norm from the exogenous input $w$ to regulated output $%
z$ is given by%
\begin{equation*}
\Phi _{wz}\left( K\right) =H+U\left[ 
\begin{array}{c}
Q \\ 
Z%
\end{array}%
\right] V.
\end{equation*}

Furthermore, the optimal achievable closed-loop gain, i.e., $\gamma
^{opt}:=\inf_{K\in \mathcal{S}\text{ stabilizing}}\left\Vert \Phi
_{wz}\left( K\right) \right\Vert $, is given by%
\begin{equation*}
\gamma ^{opt}=\inf_{Q,Z\in \mathcal{S}}\left\Vert H+U\left[ 
\begin{array}{c}
Q \\ 
Z%
\end{array}%
\right] V\right\Vert ,
\end{equation*}%
subject to that $Q$ and $Z$ satisfying the model-matching (\ref{cond1}) with 
$\varepsilon =0$, i.e. forcing $\mathcal{E=}{\normalsize 0}$ in (\ref%
{eq:epsilon}), and structure constraints (\ref{eq:QZ}).
\end{theorem}

\begin{remark}
Given Assumptions (\ref{assumption1})- (\ref{assumption2}), any stably
implementable controller has a corresponding factorization $\left(
Q,Z\right) $ satisfying $\mathcal{E=}0$. In our development, however, we
showed that we can work with the relaxed condition $\left\Vert \mathcal{E}%
\right\Vert <1$ instead of $\mathcal{E=}{\normalsize 0}$. Relaxing this
condition provides a computational advantage specially for the $l_{1}$ or $%
l_{\infty }$ problems where one typically seeks for the solution amongst
finite impulse response $Q$ and $Z$ operators. In such cases, satisfying $%
\mathcal{E=}{\normalsize 0}$ might be more challenging. Furthermore, as we
will see later, for the class of problems where Assumptions (\ref%
{assumption1})-(\ref{assumption2}) do not hold, one can still find a subset
of all stabilizing controllers by enforcing the structure on $Q$ and $Z$ as
before. However, in those cases, $\mathcal{E}$ cannot necessarily be made
arbitrarily small.
\end{remark}

Although $\mathcal{E}$ given in (\ref{eq:epsilon}) can be made equal to
zero, for computational purposes, we would be interested to find upper and
lower bounds on the closed-loop system gain when $\mathcal{E}$ is not
exactly equal to zero.

\begin{theorem}
\label{thm:synthesis}Fix $\rho _{1}\in \lbrack 0,1)$ and $\rho _{2}>0$.
Then, one can bound the optimal closed-loop norm from above and below,
respectively, by positive numbers $\gamma _{upper}$ and $\gamma _{lower}$
via the following convex optimizations:

\textbf{Upper Bound:} An upper bound on the optimal closed-loop norm, i.e., $%
\inf_{K}\left\Vert \Phi _{wz}\left( K\right) \right\Vert \leq \gamma
_{upper} $, can be obtained via convex optimization:%
\begin{equation}
\gamma _{upper}=\inf_{Q,Z,\varepsilon }\left\Vert H+U\left[ 
\begin{array}{c}
Q \\ 
Z%
\end{array}%
\right] V\right\Vert +\frac{\varepsilon \rho _{2}}{1-\rho _{1}},
\label{eq:UB}
\end{equation}%
subject to%
\begin{eqnarray}
&&\left\Vert U\right\Vert \left\Vert \left[ 
\begin{array}{c}
I+Q \\ 
Z%
\end{array}%
\right] \right\Vert \left\Vert V\right\Vert \leq \rho _{2},  \label{eq:1000}
\\
&&\varepsilon \leq \rho _{1},  \label{eq:1001} \\
&&\left\Vert \left[ 
\begin{array}{cc}
\Lambda \bar{A} & 0 \\ 
\bar{C}_{2} & 0%
\end{array}%
\right] +\left[ 
\begin{array}{ccc}
\Lambda \bar{A}-I & 0 & \Lambda \bar{B}_{2} \\ 
\bar{C}_{2} & -I & 0%
\end{array}%
\right] \left[ 
\begin{array}{c}
Q \\ 
Z%
\end{array}%
\right] \right\Vert \leq \varepsilon .  \label{eq:1002}
\end{eqnarray}

\textbf{Lower Bound:} A lower bound on the optimal closed-loop norm, i.e., $%
\inf_{K}\left\Vert \Phi _{wz}\left( K\right) \right\Vert \geq \gamma
_{lower} $, can be obtained via convex optimization:%
\begin{equation}
\gamma _{lower}=\inf_{Q,Z,\varepsilon }\left\Vert H+U\left[ 
\begin{array}{c}
Q \\ 
Z%
\end{array}%
\right] V\right\Vert +\frac{\varepsilon \rho _{2}}{1-\rho _{1}},
\label{eq:LB}
\end{equation}%
subject to (\ref{eq:1001}) and (\ref{eq:1002}).
\end{theorem}

\bigskip

\subsection{Tractable algorithm to synthesize optimal controllers with
arbitrary accuracy}

Theorem \ref{thm:synthesis} quantifies upper and lower bounds on the optimal
performance. These upper and lower bounds converge as value of the parameter 
$\rho _{2}$ grows large, as we will show next. Without loss of generality,
one can take $\rho _{2}$ to be integer valued. The algorithm is as follows:

\begin{algorithm}
\label{alg1}

Step 1(Initialization): Find a stabilizing output-feedback controller. To
this end, find a state-estimator using Theorem \ref{thm:SE}. Then, given $%
\varepsilon $ and $\rho _{1}$ with $0\leq \varepsilon \leq \rho _{1}<1$, use
Theorem \ref{thm:cond_output} to find $\left( Q,Z\right) $ such that (\ref%
{cond1})-(\ref{eq:QZ}) hold. Then, a stabilizing controller is given by (\ref%
{eq:ctrl}). Furthermore, for this pair $\left( Q,Z\right) $, pick an initial
integer value for $\rho _{2}$ such that (\ref{eq:1000}) is satisfied.

Step 2: Given $\rho _{2}$, find an upper bound, $\gamma _{upper}^{\rho _{2}}$%
, on the input-output gain by solving the convex optimization (\ref{eq:UB})
subject to constraints (\ref{eq:1000})-(\ref{eq:1002}).

Step 3: Given $\rho _{2}$, find a lower bound, $\gamma _{lower}^{\rho _{2}}$%
, on the input-output gain by solving the convex optimization (\ref{eq:LB})
subject to constraints (\ref{eq:1001}) and (\ref{eq:1002}).

Step 4: Increase $\rho _{2}$ by one unit, i.e. $\rho _{2}=\rho _{2}+1$, and
jump to Step 2. Repeat the iterations until $\gamma _{upper}^{\rho
_{2}}-\gamma _{lower}^{\rho _{2}}$ is less than the desired value.
\end{algorithm}

\begin{theorem}
Algorithm \ref{alg1} converges. That is,%
\begin{equation*}
\lim_{\rho _{2}\rightarrow \infty }\gamma _{lower}^{\rho _{2}}=\lim_{\rho
_{2}\rightarrow \infty }\gamma _{upper}^{\rho _{2}}=\inf_{K}\left\Vert \Phi
_{wz}\left( K\right) \right\Vert .
\end{equation*}
\end{theorem}

\bigskip

\section{Extension to other Control Structures}

Throughout this paper, Assumptions \ref{assumption1}-\ref{assumption2} were
made on the control structure $\mathcal{S}$. These assumptions are satisfied
if the topology of subcontrollers' network is identical to or richer than
that of the subsystems. There are, however, situations that the controller's
information structure is different than that of the generalized plant. In
such cases, Assumptions \ref{assumption1}-\ref{assumption2} are not
necessarily satisfied. In this case, for a controller to be stably
implementable over the structure we still need it to satisfy the conditions
in Definition \ref{def:stab_ctrl}. That is, the controller needs to have an
operator realization that conforms with the network structure. Therefore, we
can directly enforce the structure on the realization of the
full-information controller (\ref{eq:FI_real}) or output feedback (\ref%
{eq:realization}). Additionally, for output feedback, we need to find an
state-estimator that conforms with the structure. In the rest of this
section, we do not make Assumptions \ref{assumption1}-\ref{assumption2} and,
instead, we adopt the following assumption:

\begin{assumption}
The network structure $\mathcal{S}$ is a subspace containing the identity
and zero elements and is closed under addition. Furthermore, assume $\Lambda 
{\normalsize X}\mathcal{\in S}$ for any $X\in \mathcal{S}$.
\end{assumption}

\subsection{Full-Information Feedback}

For the full-information feedback problems, we first parametrize the set of
all centralized stabilizing controllers in the sense of Definition \ref%
{def:stab_central}. That is the set of controllers that conform with the
structure of the network $\mathcal{S}$ but the stability is guaranteed if
they are implemented in a centralized way. Next, we will show how such
controllers can be stably realized/implemented in a distributed way over the
network structure in the sense of Definition \ref{def:stab_ctrl}.

\begin{theorem}
A full-information controller $K$ is stably realizable over the network
structure $\mathcal{S}$ if there exist $\varepsilon \in \lbrack 0,1)$ and
stable operators $Q\in \mathcal{S}^{2\times 2}$ and $Z\in \mathcal{S}%
^{1\times 2}$ such that%
\begin{equation}
\left\Vert \left[ 
\begin{array}{cc}
\Lambda \bar{A} & 0 \\ 
\bar{C}_{2} & 0%
\end{array}%
\right] +\left[ 
\begin{array}{cc}
\Lambda \bar{A}-I & 0 \\ 
\bar{C}_{2} & -I%
\end{array}%
\right] Q+\left[ 
\begin{array}{c}
\Lambda \bar{B}_{2} \\ 
0%
\end{array}%
\right] Z\right\Vert \leq \varepsilon .  \label{eq:410}
\end{equation}%
In this case, a stable realization is given by%
\begin{equation*}
K:\left\{ 
\begin{array}{c}
x_{K}=\mathcal{-}{\normalsize Q}x_{K}+\left[ 
\begin{array}{c}
x \\ 
y%
\end{array}%
\right] \\ 
u={\normalsize Z}x_{K}%
\end{array}%
\right. .
\end{equation*}%
Under Assumptions \ref{assumption1}-\ref{assumption2}, (\ref{eq:410})
becomes a necessary as well as sufficient condition for stable realizability
of a controller.
\end{theorem}

\begin{proof}
The proof is similar to that of Theorems \ref{thm:state-feedback} and \ref%
{thm:FI_real}.
\end{proof}

\begin{remark}
It is important to note that, unlike Theorem \ref{thm:state-feedback} where
Assumptions \ref{assumption1}-\ref{assumption2} were made, it is not
necessarily the case that $\varepsilon $ in (\ref{eq:410}) can be made
arbitrarily small. Therefore, in the absence of Assumptions \ref{assumption1}%
-\ref{assumption2}, it is crucial to work with the relaxed model-matching (%
\ref{eq:410}) with $\varepsilon \in \lbrack 0,1)$ as opposed to fixing $%
\varepsilon =0$.
\end{remark}

\bigskip

\subsection{Output Feedback}

Our approach of finding output feedback controllers relies on a new
separation principle which utilizes a state-estimator. A state-estimator in
a generic form is given by%
\begin{equation}
\hat{x}=E_{1}u+E_{2}y,  \label{eq:411}
\end{equation}%
where $E_{1}$ and $E_{2}$ are causal operators.

\begin{lemma}
\label{lem:SE}The state-estimator (\ref{eq:411}) results in a bounded
estimation error, $e=\hat{x}-x$, if and only if for any $\varepsilon _{L}>0$
there exists stable causal operators $Q_{L}^{\varepsilon _{L}}$ and $%
Z_{L}^{\varepsilon _{L}}$ such that $\left\Vert \mathcal{E}_{L}\right\Vert
<\varepsilon _{L}$, where%
\begin{equation}
\left( I+Q_{L}^{\varepsilon _{L}}\right) \Lambda \bar{A}+Z_{L}^{\varepsilon
_{L}}\bar{C}_{2}-Q_{L}^{\varepsilon _{L}}=\mathcal{E}_{L}.  \label{eq:752}
\end{equation}%
In this case,%
\begin{eqnarray*}
E_{1} &=&\left( I-\mathcal{E}_{L}\right) ^{-1}\left( I+Q_{L}^{\varepsilon
_{L}}\right) \Lambda \bar{B}_{2}, \\
E_{2} &=&-\left( I-\mathcal{E}_{L}\right) ^{-1}Z_{L}^{\varepsilon _{L}}.
\end{eqnarray*}
\end{lemma}

The above lemma characterizes the set of all state-estimators. Such
state-estimators conform with the network structure $\mathcal{S}$ if and
only if $E_{1}$ and $E_{2}$, which are bounded operators, conform with the
network structure. When $\varepsilon _{L}=0$, in Lemma \ref{lem:SE}, $%
E_{1},E_{2}\in \mathcal{S}$ if and only if%
\begin{equation}
\left( I+Q_{L}^{\varepsilon_L}\right) \Lambda \bar{B}_{2}\in \mathcal{S}%
{\normalsize ,}Z_{L}^{\varepsilon_L}\in \mathcal{S}{\normalsize .}
\label{eq:750}
\end{equation}%
However, when $0<\varepsilon _{L}<1$, without Assumptions \ref{assumption1}-%
\ref{assumption2}, $\mathcal{E}_{L}$ may not have the network structure $%
\mathcal{S}$ and hence (\ref{eq:750}) does not, in general, yield $%
E_{1},E_{2}\in \mathcal{S}$. Therefore, for $\varepsilon _{L}\neq 0$, we
will consider an alternative observer%
\begin{equation*}
\hat{x}_{new}=\left( I-\mathcal{E}_{L}\right) \hat{x}=\left(
I+Q_{L}^{\varepsilon }\right) \Lambda \bar{B}_{2}u-Z_{L}^{\varepsilon }y.
\end{equation*}%
We note that (\ref{eq:750}) ensures that $\hat{x}_{new}$ has the network
structure and furthermore since $\hat{x}_{new}-\hat{x}=\mathcal{E}_{L}\hat{x}
$, the difference can be made arbitrarily small by choosing a small $%
\varepsilon _{L}$. Then, for the output feedback problem, we will use $\hat{x%
}_{new}$ instead of $\hat{x}$. We write the control input as%
\begin{eqnarray*}
u &=&\left[ 
\begin{array}{cc}
K_{1} & K_{2}%
\end{array}%
\right] \left[ 
\begin{array}{c}
\hat{x}_{new} \\ 
y%
\end{array}%
\right] \\
&=&\left( \left[ 
\begin{array}{cc}
K_{1} & K_{2}%
\end{array}%
\right] +\left[ 
\begin{array}{cc}
\mathcal{E}_{L} & 0 \\ 
0 & 0%
\end{array}%
\right] \right) \left[ 
\begin{array}{c}
\hat{x} \\ 
y%
\end{array}%
\right] ,
\end{eqnarray*}%
and treat $\left[ 
\begin{array}{cc}
\mathcal{E}_{L} & 0 \\ 
0 & 0%
\end{array}%
\right] $ as an additive uncertainty with size $\varepsilon _{L}$. Then,
appealing to the small-gain theorem, we have the following:

\begin{theorem}
Suppose there exists $\left( \varepsilon _{L},Q_{L}^{\varepsilon
_{L}},Z_{L}^{\varepsilon _{L}}\right) $ satisfying (\ref{eq:752}) and (\ref%
{eq:750}), and there exist bounded operators $Q=\left[ 
\begin{array}{cc}
Q_{11} & Q_{12} \\ 
Q_{21} & Q_{22}%
\end{array}%
\right] \in \mathcal{S}^{2\times 2}$ and $Z=\left[ 
\begin{array}{cc}
Z_{1} & Z_{2}%
\end{array}%
\right] \in \mathcal{S}^{1\times 2}$ such that 
\begin{eqnarray*}
\left\Vert \left[ 
\begin{array}{cc}
\Lambda \bar{A} & 0 \\ 
\bar{C}_{2} & 0%
\end{array}%
\right] +\left[ 
\begin{array}{cc}
\Lambda \bar{A}-I & 0 \\ 
\bar{C}_{2} & -I%
\end{array}%
\right] Q+\left[ 
\begin{array}{c}
\Lambda \bar{B}_{2} \\ 
0%
\end{array}%
\right] Z\right\Vert &\leq &\varepsilon , \\
\varepsilon _{L}\left\Vert I+Q\right\Vert &<&1,
\end{eqnarray*}%
for some $\varepsilon \in \lbrack 0,1)$. Then the following controller is
stabilizing and stably implementable over the network structure:%
\begin{equation}
K:\left\{ 
\begin{array}{c}
x_{K}=\mathcal{A}_{K}x_{K}+\mathcal{B}_{K}y \\ 
u=\mathcal{C}_{K}x_{K}%
\end{array}%
\right. ,
\end{equation}%
where $x_{K}=\left[ \hat{x}^{T},\xi _{1}^{T},\xi _{2}^{T},u^{T}\right] $%
\begin{eqnarray*}
\mathcal{A}_{K} &=&\left[ 
\begin{array}{cccc}
0 & 0 & 0 & \left( I+Q_{L}^{\varepsilon_L}\right) \Lambda \bar{B}_{2} \\ 
I & -Q_{11} & -Q_{12} & 0 \\ 
0 & -Q_{21} & -Q_{22} & 0 \\ 
0 & Z_{1} & Z_{2} & 0%
\end{array}%
\right] , \\
\mathcal{B}_{K} &=&\left[ 
\begin{array}{c}
-Z_{L}^{\varepsilon_L} \\ 
0 \\ 
I \\ 
0%
\end{array}%
\right] ,\mathcal{C}_{K}=\left[ 
\begin{array}{cccc}
0 & 0 & 0 & I%
\end{array}%
\right] .
\end{eqnarray*}
\end{theorem}

\section{Illustrative Example}

Consider a nested structure with two subsystems $P_{1}$ and $P_{2}$ given as
follows:%
\begin{equation*}
P_{1}:\left\{ 
\begin{array}{c}
x_{1}\left( t+1\right) =Ax_{1}\left( t\right) +B_{1}w_{1}\left( t\right)
+B_{2}u_{1}\left( t\right) \\ 
z_{1}\left( t\right) =C_{1}x_{1}\left( t\right) +D_{12}u_{1}\left( t\right)
\\ 
y_{1}\left( t\right) =C_{2}x_{1}\left( t\right) +w_{1}\left( t\right) \\ 
\nu _{21}\left( t\right) =C_{3}x_{1}\left( t\right)%
\end{array}%
\right. ,
\end{equation*}%
\begin{equation*}
P_{2}:\left\{ 
\begin{array}{c}
x_{2}\left( t+1\right) =Ax_{2}\left( t\right) +B_{2}u_{2}\left( t\right)
+B_{3}\nu _{21}\left( t-1\right) \\ 
z_{2}\left( t\right) =C_{1}x_{2}\left( t\right) +D_{12}u_{2}\left( t\right)
\\ 
y_{2}\left( t\right) =C_{2}x_{2}\left( t\right) +w_{2}\left( t\right)%
\end{array}%
\right. ,
\end{equation*}%
where%
\begin{eqnarray*}
A &=&\left[ 
\begin{array}{cc}
0.5 & 0 \\ 
0.3 & 1.2%
\end{array}%
\right] ,B_{1}\left[ 
\begin{array}{c}
1 \\ 
-0.5%
\end{array}%
\right] ,B_{2}=\left[ 
\begin{array}{c}
0.5 \\ 
1%
\end{array}%
\right] ,B_{3}=\left[ 
\begin{array}{c}
1 \\ 
0.1%
\end{array}%
\right] , \\
C_{1} &=&\left[ 
\begin{array}{cc}
1 & 2%
\end{array}%
\right] ,C_{2}=\left[ 
\begin{array}{cc}
0.1 & 1%
\end{array}%
\right] ,D_{12}=1,C_{3}=\left[ 
\begin{array}{cc}
0.2 & 0.1%
\end{array}%
\right] .
\end{eqnarray*}%
We enforce a lower-triangular structure on the aggregate controller with a
single step delay from subsystem $1$ to $2$. That is,%
\begin{equation*}
\left[ 
\begin{array}{c}
u_{1} \\ 
u_{2}%
\end{array}%
\right] =\left[ 
\begin{array}{cc}
K_{11} & 0 \\ 
\Lambda K_{21} & K_{22}%
\end{array}%
\right] \left[ 
\begin{array}{c}
y_{1} \\ 
y_{2}%
\end{array}%
\right] .
\end{equation*}%
Following our developments in this paper, to parametrize the set of all
stabilizing controllers we need to find one state-estimator. This state
estimator is given by (\ref{SE}) with $\mathcal{E}_{L}=0$ and FIR $Q_{L}$
and $Z_{L}$ as follows: $Q_{L}=\sum_{k=0}^{2}\Lambda ^{k}q_{L}\left(
k\right) $, $Z_{L}=\sum_{k=0}^{2}\Lambda ^{k}z_{L}\left( k\right) $, with%
\begin{eqnarray*}
q_{L}\left( 0\right) &=&\left[ 
\begin{array}{cccc}
0.4680 & -0.3197 & 0 & 0 \\ 
0.1685 & -0.1151 & 0 & 0 \\ 
0 & 0 & 0.4785 & -0.2148 \\ 
0 & 0 & 0.1722 & -0.0780%
\end{array}%
\right] , \\
q_{L}\left( 1\right) &=&\left[ 
\begin{array}{cccc}
0.1381 & -0.3836 & 0 & 0 \\ 
0.0497 & -0.1381 & 0 & 0 \\ 
0.0941 & -0.9590 & 0.1570 & -0.4361 \\ 
-0.0048 & -0.2376 & 0.0564 & -0.1567%
\end{array}%
\right] , \\
q_{L}\left( 2\right) &=&\left[ 
\begin{array}{cccc}
0 & 0 & 0 & 0 \\ 
0 & 0 & 0 & 0 \\ 
-0.0397 & -0.0094 & 0 & 0 \\ 
-0.0143 & -0.0034 & 0 & 0%
\end{array}%
\right] ,
\end{eqnarray*}%
\begin{eqnarray*}
z_{L}\left( 0\right) &=&\left[ 
\begin{array}{cc}
-0.3197 & 0 \\ 
-1.3151 & 0 \\ 
0 & -0.2148 \\ 
0 & -1.2780%
\end{array}%
\right] , \\
z_{L}\left( 1\right) &=&\left[ 
\begin{array}{cc}
0 & 0 \\ 
0 & 0 \\ 
-1.0590 & -0.1784 \\ 
-0.2476 & -0.0631%
\end{array}%
\right] ,z_{L}\left( 2\right) =\left[ 
\begin{array}{cc}
0.4604 & 0 \\ 
0.1657 & 0 \\ 
1.0957 & 0.5233 \\ 
0.2653 & 0.1880%
\end{array}%
\right] .
\end{eqnarray*}

We are interested in finding the $l_{\infty }$ induced optimal controller.
For this particular performance metric, the model matching problems in
Algorithm \ref{alg1} can be reduced to linear programs. This is carried out
by performing the optimization over the space of impulse responses of stable
systems $Q$ and $Z$. Furthermore, such infinite-dimensional optimizations
can be approximated from above and below by finite-dimensional optimizations
over the set of FIR $Q\ $and $Z$ of the following form $Q=\sum_{k=0}^{N}%
\Lambda ^{k}q\left( k\right) $, $Z=\sum_{k=0}^{N}\Lambda ^{k}r\left(
k\right) $, using the scaled-Q method of \cite{khammash2000new}. The finite
upper and lower finite-dimensional approximations converge to optimal
solution as $N$ grows large. For large $N$, we achieve the optimal
closed-loop $l_{\infty }$ induced norm of $1.5$ as shown in Figure \ref{fig:norm}.

\begin{figure}[h]
\center \includegraphics[width=5cm]{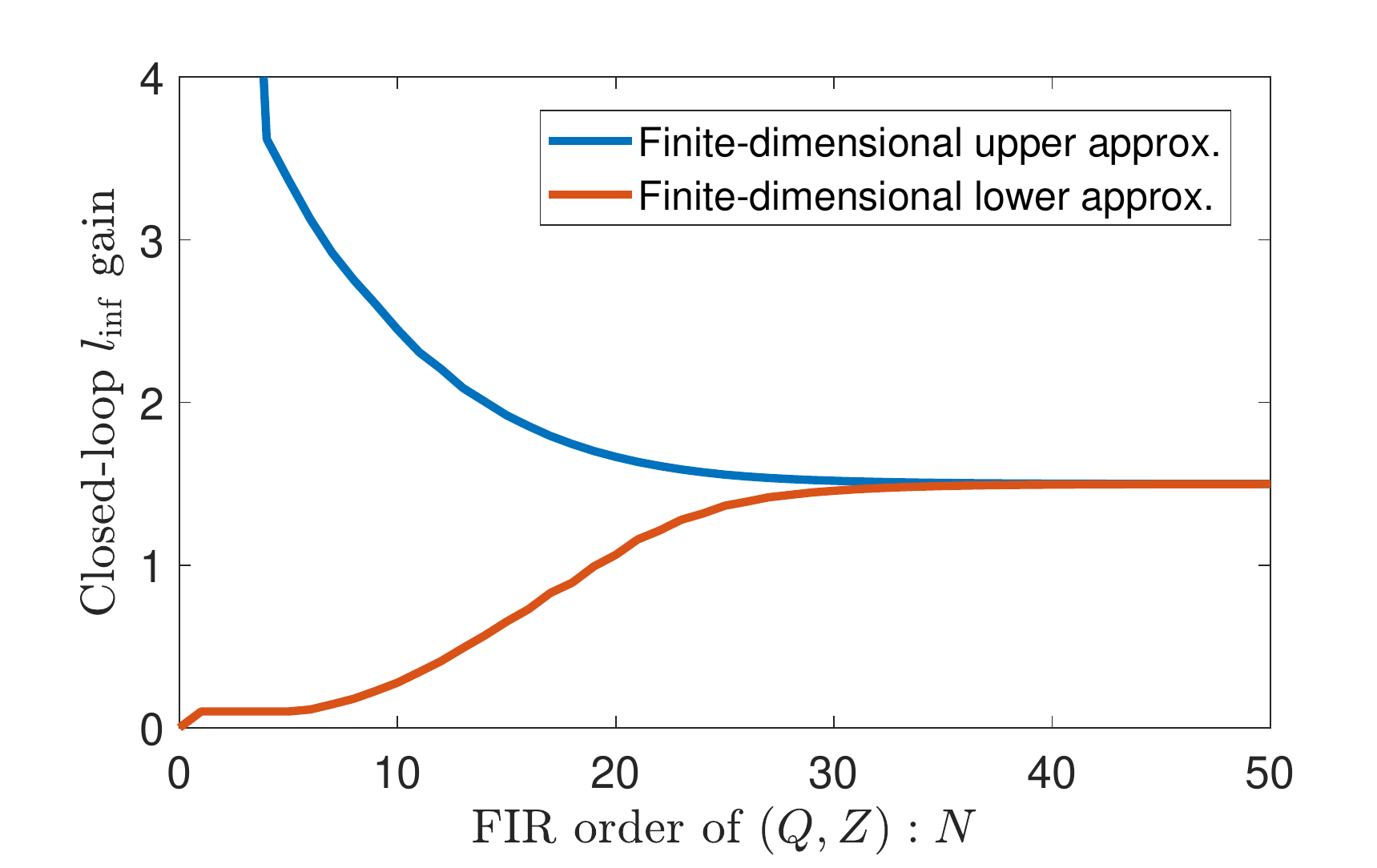}
\caption{Upper and lower bound approximations of optimal gain} \label{fig:norm}
\end{figure}

\section{Conclusion}

In this paper, we proposed a framework to synthesize structured controllers
that can be stably realized over the network. This framework is unifying in
the sense that various linear system, e.g., LTI, LTV, and linear switched
systems, can be treated analogously with respect to any measure of
performance, e.g., $l_{1}$, $l_{2}$, or $l_{\infty }$ induced norms. Our
approach is based on utilizing an operator representation of the system and
combining it with the classical Youla-parameterization. We formulated the
stability and performance problems as tractable model-matching convex
optimization. We proposed a new separation principle which was utilized to
find output feedback controllers. Furthermore, the controllers can be stably
realized over the network.

\section{Appendix}

In this section, we present the proofs of the results presented in the
paper. We will make use of the following lemmas regarding operators on $%
l_{\infty ,e}$ in the proofs.

\begin{lemma}
\label{lem:01}Given an invertible partitioned operator $X=\left[ 
\begin{array}{cc}
X_{11} & X_{12} \\ 
X_{21} & X_{22}%
\end{array}%
\right] $, its inverse is given by $X^{-1}=\left[ 
\begin{array}{cc}
S_{11} & S_{12} \\ 
S_{21} & S_{22}%
\end{array}%
\right] $ where $S_{11}=\left( X_{11}-X_{12}X_{22}^{-1}X_{21}\right) ^{-1}$, 
$S_{21}=-X_{22}^{-1}X_{21}S_{11}$, $S_{22}=\left(
X_{22}-X_{21}X_{11}^{-1}X_{12}\right) ^{-1}$, and $%
S_{12}=-X_{11}^{-1}X_{12}S_{22}$.
\end{lemma}

\begin{lemma}
\label{lem:02}Suppose $X\in \mathcal{S}$ with $\mathcal{S}$ satisfying
Assumptions \ref{assumption1} and \ref{assumption2}. Then, $\left( I-\Lambda
X\right) ^{-1}\in \mathcal{S}$.
\end{lemma}

\subsection{Proof of Lemmas \protect\ref{lemma:inverse}, \protect\ref{lem:01}%
, \protect\ref{lem:02}}

The proof of Lemma \ref{lem:01} is by straight inspection that $%
XX^{-1}=X^{-1}X=I$. The proofs of \ref{lemma:inverse} and \ref{lem:02}
follow similarly to the proofs in \cite{rotkTAC06} and omitted here in the
interest of space.

\subsection{Proof of Theorem \protect\ref{thm:state-feedback}}

We will prove this theorem in the following order $1\Rightarrow 4\Rightarrow
3\Rightarrow 2\Rightarrow 1$.

$1\Rightarrow 4$) Suppose there exists a centralized stabilizing controller $%
K$, conforming with the network structure $\mathcal{S}$, from $\left[
x^{T},y^{T}\right] ^{T}$ to $u$ such that the closed-loop system is stable.
Direct calculation verifies that the interconnection of the generalized
plant (\ref{eq:plant_general}) with $K=\left[ 
\begin{array}{cc}
K_{1} & K_{2}%
\end{array}%
\right] $ results in the closed-loop map (\ref{eq:cl'}). According to
Definition \ref{def:stab_central}, the mappings from $\bar{x}_{0}$ and $w$
to $x$, $y$, and $u$ need to be bounded. Under Assumption \ref{assumption3},
this renders operators%
\begin{equation}
Q^{0}:=\left[ 
\begin{array}{cc}
I-\Lambda \bar{A}-\Lambda \bar{B}_{2}K_{1} & -\Lambda \bar{B}_{2}K_{2} \\ 
-\bar{C}_{2} & I%
\end{array}%
\right] ^{-1}-I,  \label{Q0}
\end{equation}%
and%
\begin{eqnarray}
Z^{0} &:&=\left[ 
\begin{array}{cc}
K_{1} & K_{2}%
\end{array}%
\right] \left[ 
\begin{array}{cc}
I-\Lambda \bar{A}-\Lambda \bar{B}_{2}K_{1} & -\Lambda \bar{B}_{2}K_{2} \\ 
-\bar{C}_{2} & I%
\end{array}%
\right] ^{-1}  \label{Z0} \\
&=&\left[ 
\begin{array}{cc}
K_{1} & K_{2}%
\end{array}%
\right] \left( I+Q^{0}\right)  \notag
\end{eqnarray}%
bounded . We note that the inverse in (\ref{Q0}) exists according to Lemma %
\ref{lemma:inverse}. In terms of these operators, (\ref{eq:cl'}) can be
rewritten as%
\begin{eqnarray*}
\left[ 
\begin{array}{c}
x \\ 
y%
\end{array}%
\right] &=&\left( I+Q^{0}\right) \left\{ \left[ 
\begin{array}{c}
\Lambda \bar{B}_{1} \\ 
\bar{D}_{21}%
\end{array}%
\right] w+\left[ 
\begin{array}{c}
I \\ 
0%
\end{array}%
\right] \bar{x}_{0}\right\} , \\
u &=&Z^{0}\left\{ \left[ 
\begin{array}{c}
\Lambda \bar{B}_{1} \\ 
\bar{D}_{21}%
\end{array}%
\right] w+\left[ 
\begin{array}{c}
I \\ 
0%
\end{array}%
\right] \bar{x}_{0}\right\} .
\end{eqnarray*}%
We need to show (\ref{cond1}) holds for $Q^{0}$ and $Z^{0}$ given in (\ref%
{Q0})-(\ref{Z0}). To this end, pre-multiplying (\ref{Q0}) by $\left[ 
\begin{array}{cc}
I-\Lambda \bar{A}-\Lambda \bar{B}_{2}K_{1} & -\Lambda \bar{B}_{2}K_{2} \\ 
-\bar{C}_{2} & I%
\end{array}%
\right] $, we obtain%
\begin{equation*}
\left[ 
\begin{array}{cc}
I-\Lambda \bar{A}-\Lambda \bar{B}_{2}K_{1} & -\Lambda \bar{B}_{2}K_{2} \\ 
-\bar{C}_{2} & I%
\end{array}%
\right] \left( I+Q^{0}\right) -I=0.
\end{equation*}%
The above expression, using $Z^{0}=\left[ 
\begin{array}{cc}
K_{1} & K_{2}%
\end{array}%
\right] \left( I+Q^{0}\right) $, yields%
\begin{equation*}
\left[ 
\begin{array}{cc}
I-\Lambda \bar{A} & 0 \\ 
-\bar{C}_{2} & I%
\end{array}%
\right] \left( I+Q^{0}\right) -\left[ 
\begin{array}{c}
\Lambda \bar{B}_{2} \\ 
0%
\end{array}%
\right] Z^{0}-I=0,
\end{equation*}%
which is the same as (\ref{cond1}) for $\varepsilon =0$. We, further need to
show that $Q^{0}$ and $Z^{0}$ have the required structure given in (\ref%
{eq:QZ}). This is carried out by direct inspection of $Q^{0}$ and $Z^{0}$
given by (\ref{Q0}) and (\ref{Z0}). According to Lemma \ref{lem:01}, $Q^{0}=%
\left[ 
\begin{array}{cc}
q_{11} & q_{12} \\ 
q_{21} & q_{22}%
\end{array}%
\right] $ where%
\begin{equation*}
q_{11}=\left( I-\Lambda \left( \bar{A}+\bar{B}_{2}K_{1}+\bar{B}_{2}K_{2}\bar{%
C}_{2}\right) \right) ^{-1}-I,q_{21}=\bar{C}_{2}q_{11},
\end{equation*}%
\begin{equation*}
q_{22}=\left[ I-\bar{C}_{2}\left( I-\Lambda \bar{A}-\Lambda \bar{B}%
_{2}K_{1}\right) \Lambda \bar{B}_{2}K_{2}\right] ^{-1}-I,
\end{equation*}%
\begin{equation*}
q_{12}=\left( I-\Lambda \bar{A}-\Lambda \bar{B}_{2}K_{1}\right) ^{-1}\Lambda 
\bar{B}_{2}K_{2}q_{22}.
\end{equation*}%
Now, notice that $q_{ij}$, for $i,j=1,2$, is strictly causal and furthermore
by Applying Lemma \ref{lem:02}, $q_{i,j}\in \mathcal{S}$. This completes the
proof of ($1\Rightarrow 4$).

$4\Rightarrow 3$) Let $Q^{\varepsilon }=Q^{0}$ and $Z^{\varepsilon }=Z^{0}$.
Then (\ref{cond1}) holds for any $\varepsilon \in \lbrack 0,1)$.

$3\Rightarrow 2$) Immediate!

$2\Rightarrow 1$) Suppose there exit stable operators $Q^{\varepsilon }$ and 
$Z^{\varepsilon }$ and some $\varepsilon \in \lbrack 0,1)$ such that (\ref%
{cond1}) holds. Define the controller $K:\left[ x^{T},y^{T}\right]
^{T}\rightarrow u$ by%
\begin{equation*}
K=Z^{\varepsilon }\left( I+Q^{\varepsilon }\right) ^{-1}.
\end{equation*}%
Then, one can verify that the following identity holds%
\begin{equation*}
\left[ 
\begin{array}{cc}
I-\Lambda \bar{A}-\Lambda \bar{B}_{2}K_{1} & -\Lambda \bar{B}_{2}K_{2} \\ 
-\bar{C}_{2} & I%
\end{array}%
\right] =\left( I-\mathcal{E}_{Q^{\varepsilon },Z^{\varepsilon }}\right)
\left( I+Q^{\varepsilon }\right) ^{-1}.
\end{equation*}%
Therefore, the the closed-loop system (\ref{eq:cl'}) reduces to%
\begin{eqnarray*}
&&\left[ 
\begin{array}{c}
x \\ 
y%
\end{array}%
\right] =\left( I+Q^{\varepsilon }\right) \left( I-\mathcal{E}%
_{Q^{\varepsilon },Z^{\varepsilon }}\right) ^{-1}\left[ 
\begin{array}{c}
\Lambda \bar{B}_{1}w+\bar{x}_{0} \\ 
\bar{D}_{21}w%
\end{array}%
\right] \\
&&u=Z^{\varepsilon }\left( I-\mathcal{E}_{Q^{\varepsilon },Z^{\varepsilon
}}\right) ^{-1}\left[ 
\begin{array}{c}
\Lambda \bar{B}_{1}w+\bar{x}_{0} \\ 
\bar{D}_{21}w%
\end{array}%
\right] ,
\end{eqnarray*}%
which is a stable system since $Z^{\varepsilon }$ and $Q^{\varepsilon }$ are
stable operators and $\left( I-\mathcal{E}_{Q^{\varepsilon },Z^{\varepsilon
}}\right) ^{-1}\leq \frac{1}{1-\varepsilon }$. This completes the proof.

\subsection{Proof of Corollary \protect\ref{cor:01}}

The proof is identical to the first part of Theorem \ref{thm:state-feedback}%
. Suppose $K=\left[ 
\begin{array}{cc}
K_{1} & K_{2}%
\end{array}%
\right] $ is stabilizing. Then, $Q^{0}$ is defined as in (\ref{Q0}) and $%
Z^{0}$ in (\ref{Z0}) is given by%
\begin{equation*}
Z^{0}=\left[ 
\begin{array}{cc}
K_{1} & K_{2}%
\end{array}%
\right] \left( I+Q^{0}\right)
\end{equation*}%
or equivalently%
\begin{equation*}
\left[ 
\begin{array}{cc}
K_{1} & K_{2}%
\end{array}%
\right] =Z^{0}\left( I+Q^{0}\right) ^{-1},
\end{equation*}%
and this completes the proof.

\subsection{Proof of Theorem \protect\ref{thm:FI_real}}

According to Definition \ref{def:stab_ctrl}, in order for (\ref{eq:FI_real})
to be a stable realization, one needs to show that the interconnection of
noisy controller (\ref{eq:noisy_ctrl}) with the generalized plant results in
a stable system. The noisy controller is given by%
\begin{equation}
\bar{K}:\left\{ 
\begin{array}{c}
x_{K}=\mathcal{-}{\normalsize Q}x_{K}+\left[ 
\begin{array}{c}
x \\ 
y%
\end{array}%
\right] +n_{x} \\ 
u={\normalsize Z}x_{K}+{\normalsize n}_{u}%
\end{array}%
\right. ,  \label{eq:404}
\end{equation}%
where $n_{x}$ and $n_{u}$ stand for subcontrollers' communication noise. Then%
\begin{equation}
u=Z\left( I+Q\right) ^{-1}\left( \left[ 
\begin{array}{c}
x \\ 
y%
\end{array}%
\right] +n_{x}\right) +{\normalsize n}_{u}.  \label{eq:402}
\end{equation}%
The generalized plant (\ref{eq:plant_general}) driven by such control input
yields%
\begin{eqnarray}
\left[ 
\begin{array}{c}
x \\ 
y%
\end{array}%
\right] &=&\left( \left[ 
\begin{array}{cc}
\Lambda \bar{A} & 0 \\ 
\bar{C}_{2} & 0%
\end{array}%
\right] +\left[ 
\begin{array}{c}
\Lambda \bar{B}_{2} \\ 
0%
\end{array}%
\right] Z\left( I+Q\right) ^{-1}\right) \left[ 
\begin{array}{c}
x \\ 
y%
\end{array}%
\right]  \notag \\
&&+\left[ 
\begin{array}{c}
\Lambda \bar{B}_{2} \\ 
0%
\end{array}%
\right] Z\left( I+Q\right) ^{-1}n_{x}+\bar{n}.  \label{eq:401}
\end{eqnarray}%
where%
\begin{equation*}
\bar{n}=\left[ 
\begin{array}{c}
\Lambda \bar{B}_{1} \\ 
\bar{D}_{21}%
\end{array}%
\right] w+\left[ 
\begin{array}{c}
\Lambda \bar{B}_{2} \\ 
0%
\end{array}%
\right] n_{u}+\left[ 
\begin{array}{c}
I \\ 
0%
\end{array}%
\right] \bar{x}_{0},
\end{equation*}%
is a bounded signal. Simplifying (\ref{eq:401}), we obtain%
\begin{eqnarray}
\left[ 
\begin{array}{c}
x \\ 
y%
\end{array}%
\right] &=&\left( I+Q\right) \left( I-\mathcal{E}_{Q,Z}\right) ^{-1}\left[ 
\begin{array}{cc}
I-\Lambda \bar{A} & 0 \\ 
-\bar{C}_{2} & I%
\end{array}%
\right] n_{x}  \notag \\
&&-n_{x}+\left( I+Q\right) \left( I-\mathcal{E}_{Q,Z}\right) ^{-1}\bar{n}
\label{eq:403}
\end{eqnarray}%
where $\mathcal{E}_{Q,Z}$ is given by (\ref{eq:epsilon}). Noticing that $%
\left\Vert \mathcal{E}_{Q,Z}\right\Vert <1$, by (\ref{cond1}), guarantees
that $\left( I-\mathcal{E}_{Q,Z}\right) ^{-1}$ is a bounded operator.
Furthermore, since $\bar{n}$ and $n_{x}$ are bounded, the boundedness of $x$
and $y$ are guaranteed. Combining (\ref{eq:403}) with and (\ref{eq:404}),
the controller satisfies%
\begin{eqnarray*}
x_{K} &=&\left( I-\mathcal{E}_{Q,Z}\right) ^{-1}\left\{ \left[ 
\begin{array}{cc}
I-\Lambda \bar{A} & 0 \\ 
-\bar{C}_{2} & I%
\end{array}%
\right] n_{x}+\bar{n}\right\} , \\
u &=&Z\left( I-\mathcal{E}_{Q,Z}\right) ^{-1}\left\{ \left[ 
\begin{array}{cc}
I-\Lambda \bar{A} & 0 \\ 
-\bar{C}_{2} & I%
\end{array}%
\right] n_{x}+\bar{n}\right\} +n_{u},
\end{eqnarray*}%
and hence both the controller state $x_{K}$ and control input $u$ are
bounded signals. This completes the proof.

\bigskip

\subsection{Proof of Theorem \protect\ref{lem:100}}

The proof is immediate by noticing that an output feedback is a special case
of full-information controller $K=\left[ 
\begin{array}{cc}
K_{1} & K_{2}%
\end{array}%
\right] $ when $K_{1}=0$. Since any full-information controller can be
written as $K=Z\left( I+Q\right) ^{-1}$ where $Q$ and $Z$ satisfy any of the
equivalent conditions of Theorem \ref{thm:state-feedback}, any output
feedback can also be written as such while satisfying all of those
equivalent conditions. In addition to those conditions, however, output
feedback needs to satisfy $K_{1}=$ $Z\left( I+Q\right) ^{-1}\left[ 
\begin{array}{c}
I \\ 
0%
\end{array}%
\right] =0$.

\subsection{Proof of Theorem \protect\ref{thm:cond_output}}

We prove this theorem in the following order $1\Leftrightarrow 2\Rightarrow
4\Rightarrow 3\Rightarrow 2$

\textbf{Proof of }$1\Rightarrow 2$\textbf{)} Immediate by letting $K_{1}=0$
and $K_{2}=\tilde{K}$.

\textbf{Proof of }$2\Rightarrow 1$\textbf{)} Suppose $u=\left[ 
\begin{array}{cc}
K_{1} & K_{2}%
\end{array}%
\right] \left[ 
\begin{array}{c}
\hat{x} \\ 
y%
\end{array}%
\right] $ is stabilizing. Since the estimation $\hat{x}$ is given by (\ref%
{eq:estimation}), we have%
\begin{equation*}
u=K_{1}\hat{x}+K_{2}y=K_{1}E_{1}u+K_{1}E_{2}y+K_{2}y,
\end{equation*}%
or equivalently $u=\left( I-K_{1}E_{1}\right) ^{-1}\left(
K_{1}E_{2}+K_{2}\right) y$, if $\left( I-K_{1}E_{1}\right) $ is invertible.
We will prove later, in Theorem \ref{thm:SE}, that $E_{1}$ can be always
made strictly causal and hence, by Lemma \ref{lemma:inverse}, $\left(
I-K_{1}E_{1}\right) ^{-1}$ exists. Define $\tilde{K}=$ $\left(
I-K_{1}E_{1}\right) ^{-1}\left( K_{1}E_{2}+K_{2}\right) $. According to
Assumption \ref{assumption1} and Lemma \ref{lem:02}, $\left(
I-K_{1}E_{1}\right) ^{-1}\left( K_{1}E_{2}+K_{2}\right) \in \mathcal{S}$ and
this completes this part of the proof.

\textbf{Proof of }$2\Rightarrow 4$\textbf{) }Suppose $u=\left[ 
\begin{array}{cc}
K_{1} & K_{2}%
\end{array}%
\right] \left[ 
\begin{array}{c}
\hat{x} \\ 
y%
\end{array}%
\right] $ is stabilizing. We will show that there exists bounded operators $%
Q^{0}$ and $Z^{0}$ structured as in (\ref{eq:QZ}) such that (\ref{cond1})
holds for $\varepsilon =0$ and hence for any other $\varepsilon \in \lbrack
0,1)$.The control input can be also written as%
\begin{equation*}
u=\left[ 
\begin{array}{cc}
K_{1} & K_{2}%
\end{array}%
\right] \left[ 
\begin{array}{c}
x+e \\ 
y%
\end{array}%
\right] ,
\end{equation*}%
where $e=\hat{x}-x$ is a bounded signal. Then, the closed-loop system becomes%
\begin{eqnarray}
&&\left[ 
\begin{array}{c}
x \\ 
y%
\end{array}%
\right] =\left[ 
\begin{array}{cc}
I-\Lambda \bar{A}-\Lambda \bar{B}_{2}K_{1} & -\Lambda \bar{B}_{2}K_{2} \\ 
-\bar{C}_{2} & I%
\end{array}%
\right] ^{-1}\times  \notag \\
&&\left\{ \left[ 
\begin{array}{c}
\Lambda \bar{B}_{1} \\ 
\bar{D}_{21}%
\end{array}%
\right] w+\left[ 
\begin{array}{c}
I \\ 
0%
\end{array}%
\right] \bar{x}_{0}+\left[ 
\begin{array}{cc}
K_{1} & K_{2}%
\end{array}%
\right] \left[ 
\begin{array}{c}
e \\ 
0%
\end{array}%
\right] \right\} ,  \notag \\
u &=&\left[ 
\begin{array}{cc}
K_{1} & K_{2}%
\end{array}%
\right] \left[ 
\begin{array}{c}
x \\ 
y%
\end{array}%
\right] .
\end{eqnarray}%
According to Definition \ref{def:stab_central}, the mappings from $\bar{x}%
_{0}$ and $w$ to $x$, $y$, and $u$ need to be bounded. Making Assumption \ref%
{assumption3}, similarly to the proof of Theorem \ref{thm:state-feedback},
operators%
\begin{equation}
Q^{0}:=\left[ 
\begin{array}{cc}
I-\Lambda \bar{A}-\Lambda \bar{B}_{2}K_{1} & -\Lambda \bar{B}_{2}K_{2} \\ 
-\bar{C}_{2} & I%
\end{array}%
\right] ^{-1}-I,
\end{equation}%
and%
\begin{equation}
Z^{0}:=\left[ 
\begin{array}{cc}
K_{1} & K_{2}%
\end{array}%
\right] \left( I+Q^{0}\right) ,  \notag
\end{equation}%
are bounded. Furthermore, identical to the proof of Theorem \ref{thm:FI_real}%
, one can show that $Q=Q^{0}$ and $Z=Z^{0}$ satisfy $\mathcal{E}%
_{Q^{0},Z^{0}}=0$ and this completes the proof of this part.

\textbf{Proof of }$4\Rightarrow 3$\textbf{) }Immediate!

\textbf{Proof of }$3\Rightarrow 2$\textbf{) }We need to show that%
\begin{equation}
u=Z^{\varepsilon }\left( I+Q^{\varepsilon }\right) ^{-1}\left[ 
\begin{array}{c}
\hat{x} \\ 
y%
\end{array}%
\right]  \label{eq:406}
\end{equation}%
results in structured controller from $y$ to $u$ and stable closed-loop
system, where $Z^{\varepsilon }$ and $Q^{e}$ satisfy $\left\Vert \mathcal{E}%
_{Q^{\varepsilon },Z^{\varepsilon }}\right\Vert \leq \varepsilon $, for some 
$\varepsilon \in \lbrack 0,1)$, and structured as in(\ref{eq:QZ}).
Controller (\ref{eq:406}) can be rewritten as%
\begin{equation*}
u=Z^{\varepsilon }\left( I+Q^{\varepsilon }\right) ^{-1}\left[ 
\begin{array}{c}
x \\ 
y%
\end{array}%
\right] +Z^{\varepsilon }\left( I+Q^{\varepsilon }\right) ^{-1}\left[ 
\begin{array}{c}
e \\ 
0%
\end{array}%
\right] ,
\end{equation*}%
where $e=\hat{x}-x$ and $\left\Vert e\right\Vert \leq \delta $ for some $%
\delta \geq 0$. Then the closed loop system is given by%
\begin{eqnarray}
&&\left[ 
\begin{array}{c}
x \\ 
y%
\end{array}%
\right] =\left( I+Q^{\varepsilon }\right) \left( I-\mathcal{E}%
_{Q^{\varepsilon },Z^{\varepsilon }}\right) ^{-1}\times  \notag \\
&&\left\{ \left[ 
\begin{array}{c}
\Lambda \bar{B}_{1} \\ 
\bar{D}_{21}%
\end{array}%
\right] w+\left[ 
\begin{array}{c}
I \\ 
0%
\end{array}%
\right] \bar{x}_{0}+\left[ 
\begin{array}{c}
\Lambda \bar{B}_{2} \\ 
0%
\end{array}%
\right] Z^{\varepsilon }\left( I+Q^{\varepsilon }\right) ^{-1}\left[ 
\begin{array}{c}
e \\ 
0%
\end{array}%
\right] \right\} .  \label{eq:30}
\end{eqnarray}%
From the definition of $\mathcal{E}_{Q^{\varepsilon },Z^{\varepsilon }}$, we
have 
\begin{eqnarray*}
&&\left[ 
\begin{array}{c}
\Lambda \bar{B}_{2} \\ 
0%
\end{array}%
\right] Z^{\varepsilon }\left( I+Q^{\varepsilon }\right) ^{-1}=-\left( I-%
\mathcal{E}_{Q^{\varepsilon },Z^{\varepsilon }}\right) \left(
I+Q^{\varepsilon }\right) ^{-1} \\
&&+\left[ 
\begin{array}{cc}
I-\Lambda \bar{A} & 0 \\ 
-\bar{C}_{2} & I%
\end{array}%
\right] .
\end{eqnarray*}%
Therefore, (\ref{eq:30}) can be further simplified to%
\begin{eqnarray}
&&\left[ 
\begin{array}{c}
x \\ 
y%
\end{array}%
\right] =-\left[ 
\begin{array}{c}
e \\ 
0%
\end{array}%
\right] +\left( I+Q^{\varepsilon }\right) \left( I-\mathcal{E}%
_{Q^{\varepsilon },Z^{\varepsilon }}\right) ^{-1}\times  \notag \\
&&\left\{ \left[ 
\begin{array}{c}
\Lambda \bar{B}_{1} \\ 
\bar{D}_{21}%
\end{array}%
\right] w+\left[ 
\begin{array}{c}
I \\ 
0%
\end{array}%
\right] \bar{x}_{0}+\left[ 
\begin{array}{cc}
I-\Lambda \bar{A} & 0 \\ 
-\bar{C}_{2} & I%
\end{array}%
\right] \left[ 
\begin{array}{c}
e \\ 
0%
\end{array}%
\right] \right\} .  \label{eq:xy}
\end{eqnarray}%
From this expression, it is obvious that signals $x$ and $y$ remain bounded.
Furthermore, one can verify that 
\begin{equation}
u=Z^{\varepsilon }\left\{ \left[ 
\begin{array}{c}
\Lambda \bar{B}_{1} \\ 
\bar{D}_{21}%
\end{array}%
\right] w+\left[ 
\begin{array}{c}
I \\ 
0%
\end{array}%
\right] \bar{x}_{0}+\left[ 
\begin{array}{cc}
I-\Lambda \bar{A} & 0 \\ 
-\bar{C}_{2} & I%
\end{array}%
\right] \left[ 
\begin{array}{c}
e \\ 
0%
\end{array}%
\right] \right\} ,  \label{eq:u}
\end{equation}%
which implies $u$ is a bounded signal as well.

\subsection{Proof of Theorem \protect\ref{thm:SE}}

Suppose $K=\left[ 
\begin{array}{cc}
0 & K_{2}%
\end{array}%
\right] :\left[ 
\begin{array}{c}
x \\ 
y%
\end{array}%
\right] \rightarrow u$ is a centralized stabilizing output feedback with $%
K_{2}\in \mathcal{S}$. Therefore, the closed-loop maps $\left[ I-\Lambda 
\bar{A}-\Lambda \bar{B}_{2}K_{2}\bar{C}_{2}\right] ^{-1}$ and $\left[
I-\Lambda \bar{A}-\Lambda \bar{B}_{2}K_{2}\bar{C}_{2}\right] ^{-1}\Lambda 
\bar{B}_{2}K_{2}$ have to be a bounded operators. Define%
\begin{equation}
L:=\bar{B}_{2}K_{2}\in \mathcal{S},  \label{eq:460}
\end{equation}%
and bounded operators $Q_{L}^{0},Z_{L}^{0}\in \mathcal{S}$ 
\begin{eqnarray}
\Lambda Q_{L}^{0} &:&=\left[ I-\Lambda \bar{A}-\Lambda L\bar{C}_{2}\right]
^{-1}-I,  \label{eq:QL} \\
\Lambda Z_{L}^{0} &:&=\left[ I-\Lambda \bar{A}-\Lambda L\bar{C}_{2}\right]
^{-1}\Lambda \bar{B}_{2}K_{2}.  \notag
\end{eqnarray}%
Then it can be easily verified that%
\begin{equation*}
\Lambda L=\Lambda \bar{B}_{2}K_{2}=\left( I+\Lambda Q_{L}^{0}\right)
^{-1}\Lambda Z_{L}^{0}.
\end{equation*}%
Furthermore, from (\ref{eq:QL}), 
\begin{equation*}
\left( I+\Lambda Q_{L}^{0}\right) \left[ I-\Lambda \bar{A}-\Lambda L\bar{C}%
_{2}\right] =I,
\end{equation*}%
or equivalently $\Lambda Q_{L}^{0}\left( I-\Lambda \bar{A}\right) -\Lambda 
\bar{A}+\Lambda Z_{L}^{0}=0$. This proves that the observer gain as defined
in (\ref{eq:460}) satisfies (\ref{eq:606}) with $\varepsilon =0$. It remains
to show that such a state-estimator results in a bounded estimation error.
To this end, suppose there exists $Q_{L}^{\varepsilon }$ and $%
Z_{L}^{\varepsilon }$ such that (\ref{eq:606}) holds for some $\varepsilon $
and $L=$ $\left( I+Q_{L}^{\varepsilon }\Lambda \right)
^{-1}Z_{L}^{\varepsilon }$. Then, by adding and subtracting the term $%
\Lambda Ly$ to plant's state equation, we obtain%
\begin{equation*}
x=\Lambda \left( \bar{A}+L\bar{C}_{2}\right) x+\Lambda \left( \bar{B}_{11}+L%
\bar{D}_{21}\right) w+\Lambda \bar{B}_{2}u-\Lambda Ly+\bar{x}_{0},
\end{equation*}%
or equivalently%
\begin{equation}
x=R_{1}\left( \Lambda \left( \bar{B}_{11}+L\bar{D}_{21}\right) w+\bar{x}%
_{0}+\Lambda \bar{B}_{2}u-\Lambda Ly\right) ,  \label{eq:461}
\end{equation}%
where $R_{1}$ is defined in (\ref{eq:R1}). Also, define $R_{2}$ as in (\ref%
{eq:R2}). Then, from (\ref{eq:606}), we have%
\begin{equation*}
\left( I+\Lambda Q_{L}^{\varepsilon }\right) \bar{A}+Z_{L}^{\varepsilon }%
\bar{C}_{2}-\bar{Q}_{L}^{\varepsilon }=\mathcal{E}_{L},
\end{equation*}%
and it can be shown that%
\begin{equation*}
R_{1}=\left( I-\Lambda \mathcal{E}_{L}\right) ^{-1}\left( I+\Lambda
Q_{L}^{\varepsilon }\right) \in \mathcal{S},
\end{equation*}%
\begin{equation*}
R_{2}=\left( I-\Lambda \mathcal{E}_{L}\right) ^{-1}\Lambda
Z_{L}^{\varepsilon }\in \mathcal{S}{\normalsize ,}
\end{equation*}%
and hence both are structured and bounded. Furthermore, (\ref{eq:461}) can
be simplified to%
\begin{equation}
x=\left( R_{1}\Lambda \bar{B}_{11}+R_{2}\bar{D}_{21}\right) w+R_{1}\bar{x}%
_{0}+R_{1}\Lambda \bar{B}_{2}u-R_{2}y,  \label{eq:462}
\end{equation}%
and (\ref{LO}) is reduced to%
\begin{equation*}
\hat{x}=R_{1}\Lambda \bar{B}_{2}u-R_{2}y.
\end{equation*}%
Therefore, the estimation error is given by%
\begin{equation*}
e=\hat{x}-x=-\left( R_{1}\Lambda \bar{B}_{11}+R_{2}\bar{D}_{21}\right)
w-R_{1}\bar{x}_{0},
\end{equation*}%
which is the same as (\ref{eq:error}) and completes the proof.

\subsection{Proof of Theorem \protect\ref{thm:realization}}

We need to show that the noisy controller (\ref{eq:real_n}) results in a
stable closed-loop system. Assumed partitioned $n_{x}=\left[ n_{\hat{x}%
}^{T},n_{\xi }^{T}\right] ^{T}$. First, we will show that $\hat{x}$ entry of 
$x_{K}$ remains an estimation of plant states, $x$, with bounded error. From
(\ref{eq:real_n}), we obtain%
\begin{equation*}
\hat{x}=\Lambda \mathcal{E}_{L}\hat{x}+\left( I+\Lambda Q_{L}^{\varepsilon
}\right) \Lambda \bar{B}_{2}\left( u-n_{u}\right) -\Lambda
Z_{L}^{\varepsilon }y+n_{\hat{x}},
\end{equation*}%
where we used $u=Z\xi +n_{u}$. Then, given the plant dynamic (\ref{eq:462}),
the new estimator error, $e=\hat{x}-x$, is given by%
\begin{equation*}
e=e_{0}+\left( I-\Lambda \mathcal{E}_{L}\right) ^{-1}n_{\hat{x}}-R_{1}n_{u},
\end{equation*}%
where $e_{0}=-\left( R_{1}\Lambda \bar{B}_{11}+R_{2}\bar{D}_{21}\right)
w-R_{1}\bar{x}_{0}$. This implies that the estimation error remains bounded.
Next, we will show that the plant's states and output, i.e., signals $x$ and 
$y$, are bounded using the fact that the estimation error remains bounded.
To this end, from (\ref{eq:real_n}), the control input $u$ is given by%
\begin{eqnarray}
&&\xi =-Q\xi +\left[ 
\begin{array}{c}
\hat{x} \\ 
y%
\end{array}%
\right] +n_{\xi },  \label{eq:2010} \\
&&u=Z\xi +n_{u}.  \label{eq:2011}
\end{eqnarray}%
Consequently, using $\hat{x}=x+e$, 
\begin{eqnarray}
&&\xi =\left( I+Q\right) ^{-1}\left\{ \left[ 
\begin{array}{c}
x \\ 
y%
\end{array}%
\right] +n_{1}\right\} ,  \label{eq:2009} \\
&&u=Z\left( I+Q\right) ^{-1}\left[ 
\begin{array}{c}
x \\ 
y%
\end{array}%
\right] +Z\left( I+Q\right) ^{-1}n_{1}+Zn_{u},  \notag
\end{eqnarray}%
where $n_{1}=\left[ e^{T},0\right] ^{T}+n_{\xi }$. Then, the closed-loop
plant dynamics becomes%
\begin{eqnarray*}
&&\left[ 
\begin{array}{c}
x \\ 
y%
\end{array}%
\right] =\left\{ \left[ 
\begin{array}{cc}
\Lambda \bar{A} & 0 \\ 
\bar{C}_{2} & 0%
\end{array}%
\right] +\left[ 
\begin{array}{c}
\Lambda \bar{B}_{2} \\ 
0%
\end{array}%
\right] Z\left( I+Q\right) ^{-1}\right\} \left[ 
\begin{array}{c}
x \\ 
y%
\end{array}%
\right] \\
&&+\left[ 
\begin{array}{c}
\Lambda \bar{B}_{1}w+\bar{x}_{0}+\Lambda \bar{B}_{2}Zn_{u} \\ 
\bar{D}_{21}w%
\end{array}%
\right] +\left[ 
\begin{array}{c}
\Lambda \bar{B}_{2} \\ 
0%
\end{array}%
\right] Z\left( I+Q\right) ^{-1}\bar{n},
\end{eqnarray*}%
where $\bar{n}=n_{1}+\left( I+Q\right) n_{u}$. Using the definition of $%
\mathcal{E}_{Q,Z}$ to simplify, we obtain%
\begin{eqnarray*}
\left[ 
\begin{array}{c}
x \\ 
y%
\end{array}%
\right] &=&-\bar{n}+\left( I+Q\right) \left( I-\mathcal{E}_{Q,Z}\right)
^{-1}\times \\
&&\left\{ \left[ 
\begin{array}{c}
\Lambda \bar{B}_{1}w+\bar{x}_{0}+\Lambda \bar{B}_{2}Zn_{u} \\ 
\bar{D}_{21}w%
\end{array}%
\right] +\left[ 
\begin{array}{cc}
I-\Lambda \bar{A} & 0 \\ 
-\bar{C}_{2} & I%
\end{array}%
\right] \bar{n}\right\} .
\end{eqnarray*}%
The above expression implies that $x$ and $y$ are bounded signals.
Furthermore, from (\ref{eq:real_n}), the control input can be written as
simplified as follows: 
\begin{eqnarray*}
u &=&Z\left( I+Q\right) ^{-1}\left[ 
\begin{array}{c}
\hat{x} \\ 
y%
\end{array}%
\right] +Z\left( I+Q\right) ^{-1}n_{\xi }+Zn_{u} \\
&=&Z\left( I-\mathcal{E}_{Q,Z}\right) ^{-1}\times \\
&&\left\{ \left[ 
\begin{array}{cc}
I-\Lambda \bar{A} & 0 \\ 
-\bar{C}_{2} & I%
\end{array}%
\right] \bar{n}+\left[ 
\begin{array}{c}
I \\ 
0%
\end{array}%
\right] \bar{x}_{0}+\left[ 
\begin{array}{c}
\Lambda \bar{B}_{1} \\ 
\bar{D}_{21}%
\end{array}%
\right] w\right\} ,
\end{eqnarray*}%
which implies $u$ is also a bounded signal. Given that $x$, $y$, $u$, and $e$
are bounded signals, we will show that the states of the controller, $x_{K}$%
, are bounded. The states of the controller are $\hat{x}$ and $\xi $. Signal 
$\hat{x}$ is bounded since $\hat{x}=e+x$. Then, from the definition of $%
\mathcal{E}_{Q,Z}$, we have%
\begin{equation}
\mathcal{E}_{Q,Z}\xi =\left[ 
\begin{array}{cc}
\Lambda \bar{A} & 0 \\ 
\bar{C}_{2} & 0%
\end{array}%
\right] \left( I+Q\right) \xi +\left[ 
\begin{array}{c}
\Lambda \bar{B}_{2} \\ 
0%
\end{array}%
\right] Z\xi -Q\xi .  \label{eq:2001}
\end{equation}%
Combining (\ref{eq:2001}) with (\ref{eq:2009}) and (\ref{eq:2011}), we obtain%
\begin{equation}
\mathcal{E}_{Q,Z}\xi =\left[ 
\begin{array}{cc}
\Lambda \bar{A} & 0 \\ 
\bar{C}_{2} & 0%
\end{array}%
\right] \left\{ \left[ 
\begin{array}{c}
x \\ 
y%
\end{array}%
\right] +n_{1}\right\} +\left[ 
\begin{array}{c}
\Lambda \bar{B}_{2} \\ 
0%
\end{array}%
\right] \left( u-n_{u}\right) -Q\xi ,  \label{eq:2004}
\end{equation}%
Using the generalized plant, one can simplify (\ref{eq:2004}) to%
\begin{equation}
\mathcal{E}_{Q,Z}\xi =\left[ 
\begin{array}{c}
x \\ 
y%
\end{array}%
\right] -Q\xi +n_{2},  \label{eq:2012}
\end{equation}%
where%
\begin{equation*}
n_{2}=-\left[ 
\begin{array}{c}
\Lambda \bar{B}_{1}w+\bar{x}_{0} \\ 
\bar{D}_{21}w%
\end{array}%
\right] +\left[ 
\begin{array}{cc}
\Lambda \bar{A} & 0 \\ 
\bar{C}_{2} & 0%
\end{array}%
\right] n_{1}-\left[ 
\begin{array}{c}
\Lambda \bar{B}_{2} \\ 
0%
\end{array}%
\right] n_{u}.
\end{equation*}%
Again from (\ref{eq:2009}), we have $\xi =\left[ x^{T},y^{T}\right]
^{T}-Q\xi +n_{1}$, and right hand side of (\ref{eq:2012}) is simplified to
yield 
\begin{equation*}
\mathcal{E}_{Q,Z}\xi =\xi +n_{2}-n_{1},
\end{equation*}%
and consequently $\xi =\left( I-\mathcal{E}_{Q,Z}\right) ^{-1}\left(
n_{2}-n_{1}\right) $. Hence $\xi _{1}$ and $\xi _{2}$ are bounded signals
and the proof is complete.

\subsection{Proof of Proposition \protect\ref{prop:1}}

The proof is carried out by direct calculations. First, given $\left[
x^{T},y^{T}\right] ^{T}$ and $u$ as in (\ref{eq:xy})-(\ref{eq:u}), the
regulated output $z$ can be simplified to%
\begin{eqnarray}
&&z=\left\{ \left[ 
\begin{array}{cc}
\bar{C}_{1} & 0%
\end{array}%
\right] \left( I+Q\right) +\bar{D}_{12}Z\right\} \left( I-\mathcal{E}%
_{Q,Z}\right) ^{-1}\times  \notag \\
&&\left\{ \left[ 
\begin{array}{c}
\Lambda \bar{B}_{1} \\ 
\bar{D}_{21}%
\end{array}%
\right] w+\left[ 
\begin{array}{c}
\bar{x}_{0} \\ 
0%
\end{array}%
\right] +\left[ 
\begin{array}{cc}
I-\Lambda \bar{A} & 0 \\ 
-\bar{C}_{2} & I%
\end{array}%
\right] \left[ 
\begin{array}{c}
e \\ 
0%
\end{array}%
\right] \right\}  \label{eq:z} \\
&&+\bar{D}_{11}w-\left[ 
\begin{array}{cc}
\bar{C}_{1} & 0%
\end{array}%
\right] \left[ 
\begin{array}{c}
e \\ 
0%
\end{array}%
\right] .
\end{eqnarray}%
From Theorem \ref{thm:SE}, the estimation error is given by (\ref{eq:error}%
). Combining (\ref{eq:z}) and (\ref{eq:error}), we obtain $z=\Phi
_{wz}w+\Phi _{\bar{x}_{0}z}\bar{x}_{0}$, where%
\begin{eqnarray*}
\Phi _{wz} &=&\left\{ \left[ 
\begin{array}{cc}
\bar{C}_{1} & 0%
\end{array}%
\right] \left( I+Q\right) +\bar{D}_{12}Z\right\} \left( I-\mathcal{E}%
_{Q,Z}\right) ^{-1}\times \\
&&\left\{ \left[ 
\begin{array}{c}
\Lambda \bar{B}_{1} \\ 
\bar{D}_{21}%
\end{array}%
\right] +\left[ 
\begin{array}{c}
\Lambda \bar{A}-I \\ 
\bar{C}_{2}%
\end{array}%
\right] \left( R_{1}\Lambda \bar{B}_{1}+R_{2}\bar{D}_{21}\right) \right\} \\
&&+\bar{C}_{1}\left( R_{1}\Lambda \bar{B}_{1}+R_{2}\bar{D}_{21}\right) +\bar{%
D}_{11}.
\end{eqnarray*}%
Using identity $\left( I-\mathcal{E}\right) ^{-1}=I+\mathcal{E}\left( I-%
\mathcal{E}\right) ^{-1}$, and rearranging the terms, we obtain%
\begin{equation*}
\Phi _{wz}=H+U\left[ 
\begin{array}{c}
Q \\ 
Z%
\end{array}%
\right] V+U\left[ 
\begin{array}{c}
I+Q \\ 
Z%
\end{array}%
\right] \mathcal{E}\left( I-\mathcal{E}\right) ^{-1}V,
\end{equation*}%
where $H$, $U$, and $V$ are given in the theorem statement.

\subsection{Proof of Theorem \protect\ref{thm:synthesis}}

\subsubsection{\textbf{Proof of Upper Bound:}}

Pick $\rho _{1}\in \lbrack 0,1)$. By Theorem \ref{thm:cond_output}, any
stabilizing controller can be written as $K=Z\left( I+Q\right) ^{-1}$ with $%
\left\Vert \mathcal{E}_{Q,Z}\right\Vert \leq \varepsilon \leq \rho _{1}$.
Then, by Proposition \ref{prop:1}, 
\begin{equation*}
\left\Vert \Phi _{wz}\left( K\right) \right\Vert \leq \left\Vert H+U\left[ 
\begin{array}{c}
Q \\ 
Z%
\end{array}%
\right] V\right\Vert +\frac{\varepsilon }{1-\varepsilon }\left\Vert
U\right\Vert \left\Vert V\right\Vert \left\Vert \left[ 
\begin{array}{c}
I+Q \\ 
Z%
\end{array}%
\right] \right\Vert .
\end{equation*}%
Therefore, 
\begin{equation*}
\gamma ^{opt}\leq \inf_{Q,Z,\varepsilon }\left\Vert H+U\left[ 
\begin{array}{c}
Q \\ 
Z%
\end{array}%
\right] V\right\Vert +\frac{\varepsilon }{1-\varepsilon }\left\Vert
U\right\Vert \left\Vert V\right\Vert \left\Vert \left[ 
\begin{array}{c}
I+Q \\ 
Z%
\end{array}%
\right] \right\Vert .
\end{equation*}%
Restricting the set of $Q$ and $Z$ to those that satisfy (\ref{eq:1000}), we
obtain%
\begin{equation*}
\gamma ^{opt}\leq \inf_{Q,Z,\varepsilon }\left\Vert H+U\left[ 
\begin{array}{c}
Q \\ 
Z%
\end{array}%
\right] V\right\Vert +\frac{\varepsilon \rho _{2}}{1-\rho _{1}},
\end{equation*}%
where we used $\frac{\varepsilon }{1-\varepsilon }\leq \frac{\varepsilon }{%
1-\rho _{1}}$.

\subsubsection{\textbf{Proof of Lower Bound:}}

Given a small number $\delta >0$, let $K^{\ast }$ be a stabilizing
controller that results in a closed-loop performance in the $\delta $-ball
of optimal. By Theorem \ref{thm:cond_output}, there exists $Q^{\ast }$ and $%
Z^{\ast }$ such that $K^{\ast }=Z^{\ast }\left( I+Q^{\ast }\right) ^{-1}$,%
\begin{equation*}
\left[ 
\begin{array}{cc}
\Lambda \bar{A} & 0 \\ 
\bar{C}_{2} & 0%
\end{array}%
\right] +\left[ 
\begin{array}{cc}
\Lambda \bar{A}-I & 0 \\ 
\bar{C}_{2} & -I%
\end{array}%
\right] Q^{\ast }+\left[ 
\begin{array}{c}
\Lambda \bar{B}_{2} \\ 
0%
\end{array}%
\right] Z^{\ast }=0,
\end{equation*}%
\begin{equation*}
\Phi _{wz}\left( K^{\ast }\right) =\left\Vert H+U\left[ 
\begin{array}{c}
Q^{\ast } \\ 
Z^{\ast }%
\end{array}%
\right] V\right\Vert \leq \gamma ^{opt}+\delta .
\end{equation*}%
Given $\rho _{1},\rho _{2}>0,$ notice that $\left( Q^{\ast },Z^{\ast
}\right) $ is a feasible solution for (\ref{eq:1001}) and (\ref{eq:1002})
with $\varepsilon =0$. Therefore, we have%
\begin{equation*}
\gamma _{lower}\leq \left\Vert H+U\left[ 
\begin{array}{c}
Q^{\ast } \\ 
Z^{\ast }%
\end{array}%
\right] V\right\Vert .
\end{equation*}%
Hence, $\gamma _{lower}\leq \gamma ^{opt}+\delta $, for any arbitrary $%
\delta \geq 0$.

\subsection{Proof of Algorithm \protect\ref{alg1}}

First, we will show that the optimization in Step 2 provide a converging
upper bound on $\gamma ^{opt}$ as $\rho _{2}$ grows larger. The fact that $%
\gamma _{upper}^{m}$ is an upper bound on $\gamma ^{opt}$ comes directly
from Theorem \ref{thm:synthesis}. Now, given a small number $\delta >0$, let 
$K^{\ast }$ be a stabilizing controller that results in a closed-loop
performance in the $\delta $-ball of optimal. That is, there exists $Q^{\ast
}$ and $Z^{\ast }$ such that $K^{\ast }=Z^{\ast }\left( I+Q^{\ast }\right)
^{-1}$ such that $\Phi _{wz}\left( K^{\ast }\right) \leq \gamma
^{opt}+\delta $, where $Q^{\ast }$ and $Z^{\ast }$ satisfy $\mathcal{E=}%
{\normalsize 0}$ with $\mathcal{E}$ defined in (\ref{eq:epsilon}). Define, $%
m^{opt}$ to be an integer such that%
\begin{equation}
\left\Vert U\right\Vert \left\Vert \left[ 
\begin{array}{c}
I+Q^{\ast } \\ 
Z^{\ast }%
\end{array}%
\right] \right\Vert \left\Vert V\right\Vert \leq m^{opt}.  \label{eq:107}
\end{equation}%
Then, it is obvious that $\left( Q^{\ast },Z^{\ast }\right) $ together with $%
\varepsilon =0$ is a feasible solution for the optimization in Step $2$ if $%
\rho _{2}\geq m^{opt}$. Since $\delta $ was arbitrary, we have $\lim_{\rho
_{2}\rightarrow \infty }\gamma _{upper}^{\rho _{2}}=\gamma ^{opt}$. Next, we
will show that Step $3$ provides a converging lower bound. Pick $\delta >0$
arbitrary and $K^{\ast }$ as before. Note that $\left( Q^{\ast },Z^{\ast
}\right) $ together with $\varepsilon =0$ is a feasible solution to the
optimization in Step $3$, for any value of $\rho _{2}$. Therefore, we always
have%
\begin{eqnarray*}
\gamma _{lower} &=&\inf_{Q,Z,\varepsilon }\left\Vert H+U\left[ 
\begin{array}{c}
Q \\ 
Z%
\end{array}%
\right] V\right\Vert +\frac{\varepsilon \rho _{2}}{1-\rho _{1}} \\
&\leq &\left\Vert H+U\left[ 
\begin{array}{c}
Q^{\ast } \\ 
Z^{\ast }%
\end{array}%
\right] V\right\Vert +\frac{0\times \rho _{2}}{1-\bar{\rho}}=\gamma
^{opt}+\delta .
\end{eqnarray*}%
Therefore, $\gamma _{lower}^{\rho _{2}}\leq \gamma ^{opt}+\delta $, and
since, this holds for any arbitrary $\delta $ we must have $\gamma
_{lower}^{\rho _{2}}\leq \gamma ^{opt}$. To show that convergence of $\gamma
_{lower}^{\rho _{2}}$ to $\gamma _{upper}^{\rho _{2}}$ and consequently $%
\gamma ^{opt}$, we note that the only difference between the optimizations
in Step $2$ and $3$ is constraint (\ref{eq:1000}). This constraint becomes
inactive if $\rho _{2}\geq m^{opt}$ as given by (\ref{eq:107}). Therefore,
we have $\lim_{\rho _{2}\rightarrow \infty }\gamma _{lower}^{\rho
_{2}}=\lim_{\rho _{2}\rightarrow \infty }\gamma _{upper}^{\rho _{2}}=\gamma
^{opt}\ $and also $\gamma _{lower}^{\rho _{2}}=\gamma _{upper}^{\rho _{2}}$,
for $\rho _{2}\geq m^{opt}$.

\subsection{Proof of Lemma \protect\ref{lem:SE}}

Suppose an state-estimator in the form of (\ref{eq:411}) results in a
bounded estimation error, $e=\hat{x}-x$. Then $e=E_{1}u+E_{2}y-x$, is
bounded where $x=\left( I-\Lambda \bar{A}\right) ^{-1}\left[ \Lambda \bar{B}%
_{1}w+\Lambda \bar{B}_{2}u+\bar{x}_{0}\right] $ and $y=\bar{C}_{2}x+\bar{D}%
_{21}w$. Simplifying the error dynamics further, we obtain%
\begin{eqnarray*}
e &=&E_{1}u+\left( E_{2}\bar{C}_{2}-I\right) x+E_{2}\bar{D}_{21}w \\
&=&\left[ \left( E_{2}\bar{C}_{2}-I\right) \left( I-\Lambda \bar{A}\right)
^{-1}\Lambda \bar{B}_{1}+E_{2}\bar{D}_{21}\right] w \\
&&+\left[ \left( E_{2}\bar{C}_{2}-I\right) \left( I-\Lambda \bar{A}\right)
^{-1}\Lambda \bar{B}_{2}+E_{1}\right] u \\
&&+\left( E_{2}\bar{C}_{2}-I\right) \left( I-\Lambda \bar{A}\right) ^{-1}%
\bar{x}_{0}.
\end{eqnarray*}%
For the error to be a bounded signal, the operators $\left( E_{2}\bar{C}%
_{2}-I\right) \left( I-\Lambda \bar{A}\right) ^{-1}$ and $E_{2}$ must be
bounded. Define%
\begin{equation}
Q_{L}=-\left( E_{2}\bar{C}_{2}-I\right) \left( I-\Lambda \bar{A}\right)
^{-1}-I,Z_{L}=-E_{2},  \label{eq:470}
\end{equation}%
and pick $E_{1}=\left( I+Q_{L}\right) \Lambda \bar{B}_{2}$. Then, from (\ref%
{eq:470}), we have $\left( I+Q_{L}\right) \Lambda \bar{A}+Z_{L}\bar{C}%
_{2}-Q_{L}=0$, and the estimator and error dynamics read%
\begin{eqnarray*}
\hat{x} &=&\left( I+Q_{L}\right) \Lambda \bar{B}_{2}u+Z_{L}y, \\
e &=&-\left[ \left( I+Q_{L}\right) \Lambda \bar{B}_{1}+Z_{L}\bar{D}_{21}%
\right] w-\left( I+Q_{L}\right) \bar{x}_{0}.
\end{eqnarray*}%
The proof of converse is similar to the proof of Theorem \ref{thm:SE}.

\bibliographystyle{IEEEtran}
\bibliography{BibFile}
\vskip 0pt plus -10fil \vspace{-1 cm} 
\begin{IEEEbiographynophoto}{Mohammad Naghnaeian}
received a Ph.D. degree, in 2016, from the University of Illinois, Urbana-Champaign, IL, USA. He is currently an assistant professor at the Department of Mechanical Engineering, Clemson University. His research interests include robust and distributed control and estimation, linear switched systems, positive systems, the security of cyber-physical systems, adaptive control, and biological systems.
\end{IEEEbiographynophoto}
\begin{IEEEbiographynophoto}{Petros G. Voulgaris}
received the Diploma in Mechanical Engineering from the National Technical University, Athens, Greece, in 1986, and the S.M. and Ph.D. degrees in Aeronautics and Astronautics from the Massachusetts Institute of Technology, Cambridge, in 1988 and 1991, respectively. Since 1991, he has been with the Department of Aerospace Engineering, University of Illinois at Urbana Champaign, where he is currently a Professor. He also holds joint appointments with the Coordinated Science Laboratory, and the department of Electrical and Computer Engineering at the same university. His research interests include robust and optimal control and estimation, communications and control, networks and control, and applications of advanced control methods to engineering practice including flight control, nano-scale control, robotics, and structural control systems. Dr. Voulgaris is a recipient of the National Science Foundation Research Initiation Award (1993), the Office of Naval Research Young Investigator Award (1995) and the UIUC Xerox Award for research. He has been an Associate Editor for the IEEE Transactions on Automatic Control and the ASME Journal of Dynamic Systems, Measurement and Control. He is also a Fellow of IEEE.
\end{IEEEbiographynophoto}
\begin{IEEEbiographynophoto}{Nicola Elia} received the Laurea degree in Electrical
Engineering from the Politecnico of Turin,
Turin, Italy, in 1987, and the Ph.D. degree in
Electrical Engineering and Computer Science from
the Massachusetts Institute of Technology,
Cambridge MA,, in 1996. He is a Fellow of IEEE and the Vincentine Hermes-Luh Chair of Electrical and Computer Engineering at University of Minnesota, Minneapolis. His research interests include computational methods for controller design, communication systems with access to feedback, control with communication constraints, and networked systems. 
\end{IEEEbiographynophoto}

\end{document}